\documentclass[a4paper]{article}

\usepackage[utf8]{inputenc}
\usepackage[round]{natbib}
\usepackage[english]{babel}
\usepackage{amsfonts}
\usepackage{amsmath}
\usepackage{amsthm}
\usepackage{amssymb}
\usepackage{dsfont}
\usepackage{url}
\usepackage{hyperref}
\usepackage{tikz-cd}
\usepackage{mathpartir}
\usepackage{mathtools}
\usepackage{enumitem}
\usepackage{color}

\mathtoolsset{showonlyrefs}

\newtheorem{theorem}{Theorem}

\newtheorem{lemma}[theorem]{Lemma}
\newtheorem{proposition}[theorem]{Proposition}
\theoremstyle{remark}
\newtheorem{remark}[theorem]{Remark}
\theoremstyle{definition}
\newtheorem{definition}[theorem]{Definition}

\begin{document}

\title{An interpretation of dependent type theory in a model category of locally cartesian closed categories}

\author{Martin E. Bidlingmaier}

\date{}

\maketitle

\begin{abstract}
  Locally cartesian closed (lcc) categories are natural categorical models of extensional dependent type theory.
  This paper introduces the ``gros'' semantics in the category of lcc categories:
  Instead of constructing an interpretation in a given individual lcc category, we show that also the category of all lcc categories can be endowed with the structure of a model of dependent type theory.
  The original interpretation in an individual lcc category can then be recovered by slicing.

  As in the original interpretation, we face the issue of coherence:
  Categorical structure is usually preserved by functors only up to isomorphism, whereas syntactic substitution commutes strictly with all type theoretic structure.
  Our solution involves a suitable presentation of the higher category of lcc categories as model category.
  To that end, we construct a model category of lcc sketches, from which we obtain by the formalism of algebraically (co)fibrant objects model categories of strict lcc categories and then algebraically cofibrant strict lcc categories.
  The latter is our model of dependent type theory.
\end{abstract}

\section{Introduction}

Locally cartesian closed (lcc) categories are natural categorical models of extensional dependent type theory \citep{lcc-categories-and-type-theory}:
Given an lcc category $\mathcal{C}$, one interprets
\begin{itemize}
  \item
    contexts $\Gamma$ as objects of $\mathcal{C}$;
  \item
    (simultaneous) substitutions from context $\Delta$ to context $\Gamma$ as morphisms $f : \Delta \rightarrow \Gamma$ in $\mathcal{C}$;
  \item
    types $\Gamma \vdash \sigma$ as morphisms $\sigma : \operatorname{dom} \sigma \rightarrow \Gamma$ in $\mathcal{C}$ with codomain $\Gamma$; and
  \item
    terms $\Gamma \vdash s : \sigma$ as sections $s : \Gamma \rightleftarrows \operatorname{dom} \sigma : \sigma$ to the interpretations of types.
\end{itemize}
A context extension $\Gamma.\sigma$ is interpreted as the domain of $\sigma$.
Application of a substitution $f : \Delta \rightarrow \Gamma$ to a type $\Gamma \vdash \sigma$ is interpreted as pullback
\begin{equation}
  \begin{tikzcd}
    \operatorname{dom} \sigma[f] \arrow[r] \arrow[d, "{\sigma[f]}"] \arrow[dr, phantom, near start, "\lrcorner"] & \operatorname{dom} \sigma \arrow[d, "\sigma"] \\
    \Delta \arrow[r, "f"] & \Gamma
  \end{tikzcd}
\end{equation}
and similarly for terms $\Gamma \vdash s : \sigma$.
By definition, the pullback functors $f^* : \mathcal{C}_{/ \Gamma} \rightarrow \mathcal{C}_{/ \Delta}$ in lcc categories $\mathcal{C}$ have both left and right adjoints $\Sigma_f \dashv f^* \dashv \Pi_f$, and these are used for interpreting $\mathbf{\Sigma}$-types and $\mathbf{\Pi}$-types.
For example, the interpretation of a pair of types $\Gamma \vdash \sigma$ and $\Gamma.\sigma \vdash \tau$ is a composable pair of morphisms $\Gamma.\sigma.\tau \xrightarrow{\tau} \Gamma.\sigma \xrightarrow{\sigma} \Gamma$, and then the dependent product type $\Gamma \vdash \mathbf{\Pi}_\sigma \, \tau$ is interpreted as $\Pi_\sigma(\tau)$, which is an object of $\mathcal{C}_{/ \Gamma}$, i.e.\@ a morphism into $\Gamma$.

However, there is a slight mismatch:
Syntactic substitution is functorial and commutes strictly with type formers, whereas pullback is generally only pseudo-functorial and hence preserves universal objects only up to isomorphism.
Here functoriality of substitution means that if one has a sequence $\mathcal{E} \xrightarrow{g} \Gamma \xrightarrow{f} \Delta$ of substitutions, then we have equalities $\sigma[g][f] = \sigma[gf]$ and $s[g][f] = s[gf]$, i.e.\@ substituting in succession yields the same result as substituting with the composition.
For pullback functors, however, we are only guaranteed a natural isomorphism $f^* \circ g^* \cong (g \circ f)^*$.
Similarly, in type theory we have $(\mathbf{\Pi}_\sigma \, \tau)[f] = \mathbf{\Pi}_{\sigma[f]} \, \tau[f^+]$ (where $f^+$ denotes the weakening of $f$ along $\sigma$), whereas for pullback functors there merely exist isomorphisms $f^*(\Pi_\sigma(\tau)) \cong \Pi_{f^*(\sigma)} \, (f^+)^*(\tau)$.

In response to these problems, several notions of models with strict pullback operations were introduced, e.g.\@ categories with families (cwfs) \citep{internal-type-theory}, and coherence techniques were developed to ``strictify'' weak models such as lcc categories to obtain models with well-behaved substitution \citep{substitution-up-to-isomorphism,on-the-interpretation-of-type-theory-in-lcc-categories,the-local-universes-model}.
Thus to interpret dependent type theory in some lcc category $\mathcal{C}$, one first constructs an equivalence $\mathcal{C} \simeq \mathcal{C}^s$ such that $\mathcal{C}^s$ can be endowed with the structure of a strict model of type theory (say, cwf structure), and then interprets type theory in $\mathcal{C}^s$.

In this paper we construct cwf structure on the category of all lcc categories instead of cwf structure on some specific lcc category.
First note that the classical interpretation of type theory in an lcc category $\mathcal{C}$ is essentially an interpretation in the slice categories of $\mathcal{C}$:
\begin{itemize}
  \item
    Objects $\Gamma \in \operatorname{Ob} \mathcal{C}$ can be identified with slice categories $\mathcal{C}_{/ \Gamma}$.
  \item
    Morphisms $f : \Delta \rightarrow \Gamma$ can be identified with lcc functors $f^* : \mathcal{C}_{/ \Gamma} \rightarrow \mathcal{C}_{/ \Delta}$ which commute with the pullback functors $\Gamma^* : \mathcal{C} \rightarrow \mathcal{C}_{/ \Gamma}$ and $\Delta^* : \mathcal{C} \rightarrow \mathcal{C}_{/ \Delta}$.
  \item
    Morphisms $\sigma : \operatorname{dom} \sigma \rightarrow \Gamma$ with codomain $\Gamma$ can be identified with the objects of the slice categories $\mathcal{C}_{/ \Gamma}$.
  \item
    Sections $s : \Gamma \leftrightarrows \operatorname{dom} \sigma : \sigma$ can be identified with morphisms $1 \rightarrow \sigma$ with $1 = \mathrm{id}_\Gamma$ the terminal object in the slice category $\mathcal{C}_{/ \Gamma}$.
\end{itemize}
Removing all reference to the base category $\mathcal{C}$, we may now attempt to interpret
\begin{itemize}
  \item
    each context $\Gamma$ as a separate lcc category;
  \item
    a substitution from $\Delta$ to $\Gamma$ as an lcc functor $f : \Gamma \rightarrow \Delta$;
  \item
    types $\Gamma \vdash \sigma$ as objects $\sigma \in \operatorname{Ob} \Gamma$; and
  \item
    terms $\Gamma \vdash s : \sigma$ as morphisms $s : 1 \rightarrow \sigma$ from a terminal object $1$ to $\sigma$.
\end{itemize}
In the original interpretation, substitution in types and terms is defined by the pullback functor $f^* : \mathcal{C}_{/ \Gamma} \rightarrow \mathcal{C}_{/ \Delta}$ along a morphism $f : \Delta \rightarrow \Gamma$.
In our new interpretation, $f$ is already an lcc functor, which we simply apply to objects and morphisms of lcc categories.

The idea that different contexts should be understood as different categories is by no means novel, and indeed widespread among researchers of geometric logic; see e.g. \citet[section 4.5]{locales-and-toposes-as-spaces}.
Not surprisingly, some of the ideas in this paper have independently already been explored, in more explicit form, in \citet{au-sketches} for geometric logic.
To my knowledge, however, an interpretation of type theory along those lines, especially one with strict substitution, has never been spelled out explicitly, and the present paper is an attempt at filling this gap.

Like Seely's original interpretation, the naive interpretation in the category of lcc categories outlined above suffers from coherence issues:
Lcc functors preserve lcc structure up to isomorphism, but not necessarily up to equality, and the latter would be required for a model of type theory.

Even worse, our interpretation of contexts as lcc categories does not admit well-behaved context extensions.
Recall that for a context $\Gamma$ and a type $\Gamma \vdash \sigma$ in a cwf, a context extension consists of a context morphism $p : \Gamma.\sigma \rightarrow \Gamma$ and a term $\Gamma.\sigma \vdash v : \sigma[p]$ such that for every morphism $f : \Delta \rightarrow \Gamma$ and term $\Delta \vdash s : \sigma[f]$ there is a unique morphism $\langle f, s \rangle : \Delta \rightarrow \Gamma.\sigma$ over $\Gamma$ such that $v[\langle f, s \rangle] = s$.
In our case a context morphism is an lcc functor in the opposite direction.
Thus a context extension of an lcc category $\Gamma$ by $\sigma \in \operatorname{Ob} \Gamma$ would consist of an lcc functor $p : \Gamma \rightarrow \Gamma.\sigma$ and a morphism $v : 1 \rightarrow p(\sigma)$ in $\Gamma.\sigma$, and $(p, v)$ would have to be suitably initial.
At first sight it might seem that the slice category $\Gamma_{/ \sigma}$ is a good candidate:
Pullback along the unique map $\sigma \rightarrow 1$ defines an lcc functor $\sigma^* : \Gamma \cong \Gamma_{/ 1} \rightarrow \Gamma_{/ \sigma}$.
The terminal object of $\Gamma_{/ \sigma}$ is the identity on $\sigma$, and applying $\sigma^*$ to $\sigma$ itself yields the first projection $\mathrm{pr}_1 : \sigma \times \sigma \rightarrow \sigma$.
Thus the diagonal $d : \sigma \rightarrow \sigma \times \sigma$ is a term $\Gamma_{/ \sigma} \vdash d : \sigma^*(\sigma)$.
The problem is that, while this data is indeed universal, it is only so in the bicategorical sense (see Lemma \ref{lem:slice-is-weak-ext}):
Given an lcc functor $f : \Gamma \rightarrow \Delta$ and term $\Delta \vdash w : f(\sigma)$, we obtain an lcc functor
\begin{equation}
  \begin{tikzcd}
    \langle f, s \rangle : \Gamma_{/ \sigma} \arrow[r, "{f_{/ \sigma}}"] & \Delta_{/ f(\sigma)} \arrow[r, "s^*"] & \Delta,
  \end{tikzcd}
\end{equation}
however, $\langle f, s \rangle$ commutes with $\sigma^*$ and $f$ only up to natural isomorphism, the equation $\langle f, s \rangle(d) = s$ holds only up to this natural isomorphism, and $\langle f, s \rangle$ is unique only up to unique isomorphism.

The issue with context extensions can be understood from the perspective of comprehension categories, an alternative notion of model of type theory, as follows.
Our cwf is constructed on the opposite of $\mathrm{Lcc}$, the category of lcc categories and lcc functors.\footnote{
  Not to be confused with the category $\mathrm{Lcc}$ of lcc sketches of Definition \ref{def:lcc-axioms}; here we mean the category of fully realized lcc categories, i.e.\@ fibrant lcc sketches.
}
The corresponding comprehension category should thus consist of a Grothendieck fibration $p : \mathcal{E} \rightarrow \mathrm{Lcc}^\mathrm{op}$ and a functor $\mathcal{P} : \mathcal{E} \rightarrow (\mathrm{Lcc}^\mathrm{op})^\rightarrow$ to the arrow category of $\mathrm{Lcc}^\mathrm{op}$ such that
\begin{equation}
  \begin{tikzcd}
    \mathcal{E} \arrow[rr, "\mathcal{P}"] \arrow[dr, "p"'] & & (\mathrm{Lcc}^\mathrm{op})^{\rightarrow} \arrow[dl, "\mathrm{cod}"] \\
    & \mathrm{Lcc}^\mathrm{op}
  \end{tikzcd}
\end{equation}
commutes and $\mathcal{P}(f)$ is a pullback square for each cartesian morphism $f$ in $\mathcal{E}$.
The data of a Grothendieck fibration $p$ as above is equivalent to a (covariant) functor $\mathrm{Lcc} \rightarrow \mathrm{Cat}$ via the Grothendieck construction, and here we simply take the forgetful functor.
Thus the objects of $\mathcal{E}$ are pairs $(\Gamma, \tau)$ such that $\Gamma$ is an lcc category and $\tau \in \operatorname{Ob} \Gamma$, and a morphism $(\Gamma, \tau) \rightarrow (\Delta, \sigma)$ in $\mathcal{E}$ is a pair $(f, k)$ of lcc functor $f : \Delta \rightarrow \Gamma$ and morphism $k : \tau \rightarrow f(\sigma)$ in $\Gamma$.

The functor $\mathcal{P}$ should assign to objects $(\Gamma, \tau)$ of $\mathcal{E}$ the projection of the corresponding context extension, hence we define $\mathcal{P}(\Gamma, \tau) = \tau^* : \Gamma \rightarrow \Gamma_{/ \tau}$ as the pullback functor to the slice category.
The cartesian morphisms of $\mathcal{E}$ are those with invertible vertical components $k$, so they are given up to isomorphism by pairs of the form $(f, \mathrm{id}) : (\Gamma, f(\sigma)) \rightarrow (\Delta, \sigma)$.
The images of such morphisms under $\mathcal{P}$ are squares
\begin{equation}
  \label{eq:homotopy-pushout-comprehension}
  \begin{tikzcd}
    \Delta \arrow[d, "\sigma^*"] \arrow[r, "f"] & \Gamma \arrow[d, "f(\tau)^*"] \\
    \Delta_{/ \sigma} \arrow[r, "f_{/ \sigma}"] & \Gamma_{/ f(\tau)}
  \end{tikzcd}
\end{equation}
in $\mathrm{Lcc}$.
For $(p, \mathcal{P})$ to be a comprehension category, they would have to be pushout squares, but in fact they are bipushout squares:
They satisfy the universal property of pushouts up to unique isomorphism, but not up to equality.
If $f$ does not preserve pullback squares up to strict equality, then the square \eqref{eq:homotopy-pushout-comprehension} commutes only up to isomorphism, not equality.
Thus $\mathcal{P}$ is not even a functor but a bifunctor.

Usually when one considers coherence problems for type theory, the problem lies in the fibration $p$, which is often not strict, and it suffices to change the total category $\mathcal{E}$ while leaving the base unaltered.
Our fibration $p$ is already strict, but it does not support substitution stable type constructors.
Here the main problem is the base category, however:
The required pullback diagrams exist only in the bicategorical sense.
Thus the usual constructions \citep{on-the-interpretation-of-type-theory-in-lcc-categories,the-local-universes-model} are not applicable.

The goal must thus be to find a category that is (bi)equivalent to $\mathrm{Lcc}$ in which we can replace the bipushout squares \eqref{eq:homotopy-pushout-comprehension} by 1-categorical pushouts.
Our tool of choice to that end will be \emph{model category theory} (see e.g.\@ \citet{hirschhorn}).
Model categories are presentations of higher categories as ordinary 1-categories with additional structure.
Crucially, model categories allow the computation of higher (co)limits as ordinary 1-categorical (co)limits under suitable hypotheses.
The underlying 1-category of the model category presenting a higher category is not unique, and some presentations are more suitable for our purposes than others.
We explore three Quillen equivalent model categories, all of which encode the same higher category of lcc categories, and show that the third one indeed admits the structure of a model of dependent type theory.

Because of its central role in the paper, the reader is thus expected to be familiar with some notions of model category theory.
We make extensive use of the notion of algebraically (co)fibrant object in a model category \citep{algebraic-models,coalgebraic-models}, but the relevant results are explained where necessary and can be taken as black boxes for the purpose of this paper.
Because of the condition on enrichment in Theorem \ref{th:coalgebraic-model-category}, all model categories considered here are proved to be model $\mathrm{Gpd}$-categories, that is, model categories enriched over the category of groupoids with their canonical model structure.
See \citet{enriched-model-cats} for background on enriched model category theory, \citet{groupoid-model-cat} for the canonical model category of groupoids, and \citet{homotopy-theoretic-aspects} for the closely related model $\mathrm{Cat}$-categories.
While it is more common to work with the more general simplicially enriched model categories, the fact that the higher category of lcc categories is 2-truncated affords us to work with simpler groupoid enrichments instead.

In Section \ref{sec:lcc-sketches} we construct the model category $\mathrm{Lcc}$ of \emph{lcc sketches}, a left Bousfield localization of an instance of Isaev's model category structure on marked objects \citep{marked-objects}.
Lcc sketches are to lcc categories as finite limit sketches are to finite limit categories.
Thus lcc sketches are categories with some diagrams marked as supposed to correspond to a universal object of lcc categories, but marked diagrams do not have to actually satisfy the universal property.
The model category structure is set up such that every lcc sketch generates an lcc category via fibrant replacement, and lcc sketches are equivalent if and only if they generate equivalent lcc categories.

In Section \ref{sec:slcc} we define the model category $\mathrm{sLcc}$ of \emph{strict lcc categories}.
Strict lcc categories are the algebraically fibrant objects of $\mathrm{Lcc}$, that is, they are objects of $\mathrm{Lcc}$ \emph{equipped} with canonical lifts against trivial cofibrations witnessing their fibrancy in $\mathrm{Lcc}$.
Such canonical lifts correspond to canonical choices of universal objects in lcc categories, and the morphisms in $\mathrm{sLcc}$ preserve these canonical choices not only up to isomorphism but up to equality.

Section \ref{sec:algebraically-cofibrant} finally establishes the model of type theory in the opposite of $\operatorname{Coa} \mathrm{sLcc}$, the model category of \emph{algebraically cofibrant objects} in $\mathrm{sLcc}$.
The objects of $\operatorname{Coa} \mathrm{sLcc}$ are strict lcc categories $\Gamma$ such that every (possibly non-strict) lcc functor $\Gamma \rightarrow \Delta$ has a canonical strict isomorph.
This additional structure is crucial to reconcile the context extension operation, which is given by freely adjoining a morphism to a strict lcc category, with taking slice categories.

In Section \ref{sec:applications} we show that the cwf structure on $(\operatorname{Coa} \mathrm{sLcc})^\mathrm{op}$ can be used to rectify Seely's original interpretation in a given lcc category $\mathcal{C}$.
This is done by choosing an equivalent lcc category $\Gamma \simeq \mathcal{C}$ with $\Gamma \in \operatorname{Ob} \, (\operatorname{Coa} \mathrm{sLcc})$, and then $\Gamma$ inherits cwf structure from the core of the slice cwf $(\operatorname{Coa} \mathrm{sLcc})^\mathrm{op}_{ / \Gamma}$.

\textbf{Acknowledgements.}
This paper benefited significantly from input by several members of the research community.
I would like to thank the organizers of the TYPES workshop 2019 and the HoTT conference 2019 for giving me the opportunity to present preliminary versions of the material in this paper.
Conversations with Emily Riehl, Karol Szumiło and David White made me aware that the constructions in this paper can be phrased in terms of model category theory.
Daniel Gratzer pointed out to me the biuniversal property of slice categories.
Valery Isaev explained to me some aspects of the model category structure on marked objects.
I would like to thank my advisor Bas Spitters for his advice on this paper, which is part of my PhD research.

This work was supported by the Air Force Office and Scientific Research project ``Homotopy Type Theory and Probabilistic Computation'', grant number 12595060.

\section{Lcc sketches}
\label{sec:lcc-sketches}

This section is concerned with the model category $\mathrm{Lcc}$ of lcc sketches.
$\mathrm{Lcc}$ is constructed as the left Bousfield localization of a model category of lcc-marked objects, an instance of Isaev's model category structure on marked objects.

\begin{definition}[\citet{marked-objects} Definition 2.1]
  Let $\mathcal{C}$ be a category and let $i : I \rightarrow \mathcal{C}$ be a diagram in $\mathcal{C}$.
  An \emph{($i$-)marked object} is given by an object $X$ in $\mathcal{C}$ and a subfunctor $m_X$ of $\mathrm{Hom}(i(-), X) : I^\mathrm{op} \rightarrow \mathrm{Set}$.
  A map of the form $k : i(K) \rightarrow X$ is \emph{marked} if $k \in m_X(K)$.

  A morphism of $i$-marked objects is a marking-preserving morphism of underlying objects in $\mathcal{C}$, i.e.\@ a morphism $f : X \rightarrow Y$ such that the image of $m_X$ under postcomposition by $f$ is contained in $m_Y$.
  The category of $i$-marked objects is denoted by $\mathcal{C}^i$.
\end{definition}

The forgetful functor $U : \mathcal{C}^i \rightarrow \mathcal{C}$ has a left and right adjoint:
Its left adjoint $X \mapsto X^\flat$ is given by equipping an object $X$ of $\mathcal{C}$ with the minimal marking $m_{X^\flat} = \emptyset \subseteq \mathrm{Hom}(i(-), X)$, while the right adjoint $X \mapsto X^\sharp$ equips objects with their maximal marking $m_{X^\sharp} = \mathrm{Hom}(i(-), X)$.

In our application, $\mathcal{C} = \mathrm{Cat}$ is the category of (sufficiently small) categories, and $I = I_\mathrm{lcc}$ contains diagrams corresponding to the shapes (e.g. a squares for pullbacks) of lcc structure.

\begin{definition}
  The subcategory $I_\mathrm{lcc} \subseteq \mathrm{Cat}$ of \emph{lcc shapes} is given as follows:
  Its objects are the three diagrams $\mathrm{Tm}$, $\mathrm{Pb}$ and $\mathrm{Pi}$.
  $\mathrm{Tm}$ is given by the category with a single object $t$ and no nontrivial morphisms; it corresponds to terminal objects.
  $\mathrm{Pb}$ is the free-standing non-commutative square
  \begin{equation}
    \begin{tikzcd}
      \cdot \arrow[d, "p_1"] \arrow[r, "p_2"] & \cdot \arrow[d, "f_2"] \\
      \cdot \arrow[r, "f_1"] & \cdot;
    \end{tikzcd}
  \end{equation}
  and corresponds to pullback squares.
  $\mathrm{Pi}$ is the free-standing non-commutative diagram
  \begin{equation}
    \begin{tikzcd}
      \cdot \arrow[ddr, bend right, "p_1"] \arrow[drr, bend left, "p_2"] \arrow[dr, "\varepsilon"] \\
      & \cdot \arrow[d, "g"] & \cdot \arrow[d, "f_2"] \\
      & \cdot \arrow[r, "f_1"] & \cdot
    \end{tikzcd}
  \end{equation}
  and corresponds to dependent products $f_2 = \Pi_{f_1}(g)$ and their evaluation maps $\varepsilon$.
  The only nontrivial functor in $I_\mathrm{lcc}$ is the inclusion of $\mathrm{Pb}$ into $\mathrm{Pi}$ as indicated by the variable names.
  It corresponds to the requirement that the domain of the evaluation map of dependent products must be a suitable pullback.

  We obtain the category $\mathrm{Cat}^\mathrm{lcc} = \mathrm{Cat}^{I_\mathrm{lcc}}$ of \emph{lcc-marked categories}.
\end{definition}

Now suppose that $\mathcal{C} = \mathcal{M}$ is a model category.
Let $\gamma : \mathcal{M} \rightarrow \operatorname{Ho} \mathcal{M}$ be the quotient functor to the homotopy category.
A marking $m_X \subseteq \mathrm{Hom}(i(-), X)$ of some $X \in \operatorname{Ob} \mathcal{M}$ induces a canonical marking $\gamma(m_X) \subseteq \mathrm{Hom}(\gamma(i(-)), \gamma(X))$ on $\gamma(X)$ by taking $\gamma(m_X)$ to be the image of $m_X$ under $\gamma$.
Thus a morphism $K \rightarrow X$ in $\operatorname{Ho} \mathcal{M}$ is marked if and only if it has a preimage under $\gamma$ which is marked.

\begin{theorem}[\citet{marked-objects} Theorem 3.3]
  \label{th:marked-model-category}
  Let $\mathcal{M}$ be a combinatorial model category and let $i : I \rightarrow \mathcal{M}$ be a diagram in $\mathcal{M}$ such that every object in the image of $i$ is cofibrant.
  Then the following defines the structure of a combinatorial model category on $\mathcal{M}^i$:
  \begin{itemize}
    \item
      A morphism $f : (m_X, X) \rightarrow (m_Y, Y)$ in $\mathcal{M}^i$ is a cofibration if and only if $f : X \rightarrow Y$ is a cofibration in $\mathcal{M}$.
    \item
      A morphism $f : (m_X, X) \rightarrow (m_Y, Y)$ in $\mathcal{M}^i$ is a weak equivalence if and only if $\gamma(f) : (\gamma(m_X), \gamma(X)) \rightarrow (\gamma(m_Y), \gamma(Y))$ is an isomorphism in $(\operatorname{Ho} \mathcal{M})^{\gamma i}$.
  \end{itemize}
  A marked object $(X, m_X)$ is fibrant if and only if $X$ is fibrant in $\mathcal{M}$ and the markings of $X$ are stable under homotopy; that is, if $k \simeq h : i(K) \rightarrow X$ are homotopic maps in $\mathcal{M}$ and $k$ is marked, then $h$ is marked.
  The adjunctions $(-)^\flat \dashv U$ and $U \dashv (-)^\sharp$ are Quillen adjunctions.
\end{theorem}

\begin{remark}
  The description of weak equivalences in Theorem \ref{th:marked-model-category} does not appear as stated in \citet{marked-objects}, but follows easily from results therein.
  Let $t : \mathrm{Id} \Rightarrow R : \mathcal{M} \rightarrow \mathcal{M}$ be a fibrant replacement functor.
  By \citet[lemma 2.5]{marked-objects}, a map $f : (m_X, X) \rightarrow (m_Y, Y)$ is a weak equivalence in $\mathcal{M}^i$ if and only if $f$ is a weak equivalence in $\mathcal{M}$ and for every diagram (of solid arrows)
  \begin{equation}
    \label{eq:isaev-classification-weak-equivalences}
    \begin{tikzcd}
      i(K) \arrow[drr, bend left, "k"] \arrow[ddr, bend right, "h"'] \arrow[dr, dashed, "h'"] \arrow[ddr, phantom, "\simeq"] \\
      & X \arrow[d, "t_X"] \arrow[r, "f"] & Y \arrow[d, "t_Y"]  \\
      & R(X) \arrow[r, "R(f)"] & R(Y)
    \end{tikzcd}
  \end{equation}
  in which the outer square commutes up to homotopy and $k$ is marked, there exists a marked map $h' : i(K) \rightarrow X$ as indicated such that $t_X h' \simeq h$.
  ($h'$ is not required to commute with $k$ and $f$.)

  Now assume that $f : (m_X, X) \rightarrow (m_Y, Y)$ satisfies this condition and let us prove that $\gamma(f)$ is an isomorphism of induced marked objects in the homotopy category.
  $\gamma(f)$ is an isomorphism in $\operatorname{Ho} \mathcal{M}$, so it suffices to show that $\gamma(f)^{-1}$ preserves markings.
  By definition, every marked morphism of $\gamma(Y)$ is of the form $\gamma(k) : \gamma(i(K)) \rightarrow \gamma(Y)$ for some marked $k : i(K) \rightarrow Y$.
  Because $i(K)$ is cofibrant and $R(X)$ is fibrant, the map $\gamma(t_X) \circ \gamma(f)^{-1} \circ \gamma(k) : \gamma(i(K)) \rightarrow \gamma(R(X))$ has a preimage $h : i(K) \rightarrow R(X)$ under $\gamma$.
  % This is \cite[Theorem 1.2.10]{hovey-model-categories}
  As $i(K)$ is cofibrant, $R(Y)$ is fibrant and $\gamma(R(f) \circ h) = \gamma(t_Y k)$, there is a homotopy $h \circ R(f) \simeq t_Y k$.
  By assumption, there exists a marked map $h' : i(K) \rightarrow X$ such that $h' t_X \simeq h$, thus $\gamma(f)^{-1} \circ \gamma(k) = \gamma(h')$ is marked.

  To prove the other direction of the equivalence, assume that $\gamma(f)$ is an isomorphism of marked objects and let $h, k$ be as in diagram \eqref{eq:isaev-classification-weak-equivalences}.
  $\gamma(f)^{-1} \gamma(k)$ is marked, hence has a preimage $h' : i(K) \rightarrow X$ under $\gamma$ which is marked.
  We have $\gamma(t_X h') = \gamma(h)$ because postcomposition of both sides with the isomorphism $\gamma(R(f))$ gives equal results.
  $i(K)$ is cofibrant and $R(X)$ is fibrant, thus $t_X h' \simeq h$.
\end{remark}

\begin{lemma}
  \label{lem:simplicial-marked-objects}
  Let $\mathcal{M}$ and $i : K \rightarrow \mathcal{M}$ be as in Theorem \ref{th:marked-model-category}.
  \begin{enumerate}[label={(\arabic*)}]
    \item
      \label{itm:marked-left-proper}
      If $\mathcal{M}$ is a left proper model category, then $\mathcal{M}^i$ is a left proper model category.
    \item
      \label{itm:marked-simplicial}
      If $\mathcal{M}$ is a model $\mathrm{Gpd}$-category, then $\mathcal{M}^i$ admits the structure of a model $\mathrm{Gpd}$-category such that $(-)^\flat \dashv U$ and $U \dashv (-)^\sharp$ lift to Quillen $\mathrm{Gpd}$-adjunctions.
  \end{enumerate}
\end{lemma}
\begin{proof}
  \emph{\ref{itm:marked-left-proper}}.
  Let 
  \begin{equation}
    \begin{tikzcd}
      X \arrow[r, "g"] \arrow[d, "f"] \arrow[dr, phantom, near end, "\ulcorner"] & Y_2 \arrow[d, "f'"] \\
      Y_1 \arrow[r, "g'"] & Z
    \end{tikzcd}
  \end{equation}
  be a pushout square in $\mathcal{M}^i$ such that $f$ is a weak equivalence.
  $\mathcal{M}$ is left proper, so $\gamma(f')$ is invertible as a map in $\operatorname{Ho} \mathcal{M}$.

  A map $k : i(K) \rightarrow U(Z)$ is marked if and only if it factors via a marked map $k_1 : i(K) \rightarrow U(Y_1)$ or via a marked map $k_2 : i(K) \rightarrow U(Y_2)$.
  In the first case,
  \begin{equation}
    \gamma(f')^{-1} \circ \gamma(k) = \gamma(g) \circ \gamma(f)^{-1} \circ \gamma(k_1),
  \end{equation}
  which is marked because $f$ is a weak equivalence.
  Otherwise
  \begin{equation}
    \gamma(f')^{-1} \gamma(k) = \gamma(k_2),
  \end{equation}
  which is also marked.
  We have shown that $\gamma(f')$ is an isomorphism of marked objects in $\operatorname{Ho} \mathcal{M}$, thus $f'$ is a weak equivalence.

  \emph{\ref{itm:marked-simplicial}}.
  Let $X$ and $Y$ be marked objects.
  We define the mapping groupoid $\mathcal{M}^i(X, Y)$ as the full subgroupoid of $\mathcal{M}(U(X), U(Y))$ of marking preserving maps.

  $\mathcal{M}^i$ is complete and cocomplete as a 1-category.
  Thus if we construct tensors $\mathcal{G} \otimes X$ and powers $X^\mathcal{G}$ for all $X \in \operatorname{Ob} \mathcal{M}$ and $\mathcal{G} \in \operatorname{Ob} \mathrm{Gpd}$ it follows that $\mathcal{M}^i$ is also complete and cocomplete as a $\mathrm{Gpd}$-category.
  The underlying object of powers and copowers is constructed in $\mathcal{M}$, i.e.\@ $G(\mathcal{G} \otimes X) = \mathcal{G} \otimes G(X)$ and $G(X^\mathcal{G}) = G(X)^\mathcal{G}$.
  A map $k : i(K) \rightarrow X^\mathcal{G}$ is marked if and only if the composite
  \begin{equation}
    \begin{tikzcd}
      i(K) \arrow[r, "k"] & G(X)^\mathcal{G} \arrow[r, "X^v"] & G(X)^1 = G(X)
    \end{tikzcd}
  \end{equation}
  is marked for every $v \in \operatorname{Ob} \mathcal{G}$ (which we identify with a map $v : 1 \rightarrow \mathcal{G}$).
  Similarly, a map $k : i(K) \rightarrow \mathcal{G} \otimes X$ is marked if and only if it factors as
  \begin{equation}
    \begin{tikzcd}
      i(K) \arrow[r, "k_0"] & G(X) = 1 \otimes G(X) \arrow[r, "v \mathbin{\otimes} \mathrm{id}"] & \mathcal{G} \otimes G(X)
    \end{tikzcd}
  \end{equation}
  for some object $v$ in $\mathcal{G}$ and marked $k_0$.
  It follows by \citet[Theorem 4.85]{basic-concepts-of-enriched-category-theory} from the preservation of tensors and powers by $U$ that the 1-categorical adjunctions $(-)^\flat \dashv U$ and $U \dashv (-)^\flat$ extend to $\mathrm{Gpd}$-adjunctions.

  It remains to show that the tensoring $\mathrm{Gpd} \times \mathcal{M}^i \rightarrow \mathcal{M}^i$ is a Quillen bifunctor.
  For this we need to prove that if $f : \mathcal{G} \rightarrowtail \mathcal{H}$ is a cofibration of groupoids and $g : X \rightarrowtail Y$ is a cofibration of marked objects, then their pushout-product
  \begin{equation}
    f \mathbin{\square} g : \mathcal{G} \otimes Y \amalg_{\mathcal{G} \otimes X} \mathcal{H} \otimes X \rightarrow \mathcal{H} \otimes Y
  \end{equation}
  is a cofibration, and that it is a weak equivalence if either $f$ or $g$ is furthermore a weak equivalence.
  The first part follows directly from the same property for the $\mathrm{Gpd}$-enrichment of $\mathcal{M}$ and the fact that $U$ preserves tensors and pushouts, and reflects cofibrations.

  In the second part we have in both cases that $f \mathbin{\square} g$ is a weak equivalence in $\mathcal{M}$.
  Thus we only need to show that $f \mathbin{\square} g$ reflects a given marked morphism $k : i(K) \rightarrow \mathcal{H} \otimes G(Y)$ in $\operatorname{Ho} \mathcal{M}$.
  It follows from the construction of $\mathcal{H} \otimes Y$ that for any such $k$ there exists $w \in \operatorname{Ob} \mathcal{H}$ such that $k = (w \otimes \mathrm{id}) \circ k_0$ for some marked map $k_0 : i(K) \rightarrow G(Y)$.

  Assume first that $f$ is a trivial cofibration, i.e.\@ an equivalence of groupoids that is injective on objects.
  Then there exists $v \in \operatorname{Ob} \mathcal{G}$ such that $f(v)$ and $w$ are isomorphic objects of $\mathcal{H}$.
  $k$ is (left) homotopic to $(f(v) \otimes \mathrm{id}) \circ k_0$, which factors via the marked map $(v \otimes \mathrm{id}) \circ k_0 : i(K) \rightarrow \mathcal{G} \otimes G(Y)$.
  It follows that $\gamma(f \mathbin{\square} g)$ reflects marked morphisms.

  Now assume that $g$ is a trivial cofibration.
  Then $\gamma(g)$ reflects marked maps, i.e.\@ there exists a marked map $h_0 : i(K) \rightarrow G(X)$ such that $\gamma(g) \circ \gamma(h_0) = \gamma(h)$.
  Thus the equivalence class of $(w \otimes \mathrm{id}) \circ h_0 : i(K) \rightarrow \mathcal{H} \otimes X$ in $\operatorname{Ho} \mathcal{M}$ is marked and mapped to $\gamma(k)$ under postcomposition by $\gamma(f)$.
\end{proof}

In the semantics of logic, one usually defines the notion of model of a logical theory in two steps:
First a notion of structure is defined that interprets the theory's signature, i.e.\@ the function and relation symbols that occur in its axioms.
Then one defines what it means for such a structure to satisfy a formula over the signature, and a model is a structure of the theory's signature which satisfies the theory's axioms.
For very well-behaved logics such as Lawvere theories, there is a method of freely turning structures into models of the theory, so that the category of models is a reflective subcategory of the category of structures.

By analogy, lcc-marked categories correspond to the structures of the signature of lcc categories.
The model structure of $\mathrm{Cat}^{\mathrm{lcc}}$ ensures that markings respect the homotopy theory of $\mathrm{Cat}$, in that the choice of marking is only relevant up to isomorphism of diagrams.
However, the model structure does not encode the universal property that marked diagrams are ultimately supposed to satisfy.
To obtain the analogue of the category of models for a logical theory, we now define a reflective subcategory of $\mathrm{Cat}^{\mathrm{lcc}}$.
The technical tool to accomplish this is a left Bousfield localization at a set $S$ of morphisms in $\mathrm{Cat}^{\mathrm{lcc}}$.
$S$ corresponds to the set of axioms of a logical theory.
We thus need to define $S$ in such a way that an lcc-marked category is lcc if and only if it has the right lifting property against the morphisms in $S$ such that lifts are determined uniquely up to unique isomorphism.

$\mathrm{Cat}$ is a combinatorial and left proper model $\mathrm{Gpd}$-category with mapping groupoids $\mathrm{Cat}(\mathcal{C}, \mathcal{D})$ given by sets of functors and their natural isomorphisms.
Thus $\mathrm{Cat}^{\mathrm{lcc}}$ has the structure of a combinatorial and left proper model $\mathrm{Gpd}$-category by Lemma \ref{lem:simplicial-marked-objects}.
It follows that the left Bousfield localization at any (small) set of maps exists by \citet[Theorem 4.1.1]{hirschhorn}.

\begin{definition}
  \label{def:lcc-axioms}
  The model category $\mathrm{Lcc}$ of \emph{lcc sketches} is the left Bousfield localization of the model category of $\mathrm{lcc}$-marked categories at the following morphisms.
  \begin{itemize}
    \item
      The morphism $\mathrm{tm}_1$ given by the unique map from the empty category to the marked category with a single, $\mathrm{Tm}$-marked object.
      $\mathrm{tm}_1$ corresponds to the essentially unique existence of a terminal object.
    \item
      The morphism $\mathrm{tm}_2$ given by the inclusion of the category with two objects
      \begin{equation}
        \begin{tikzcd}
          \cdot & t
        \end{tikzcd}
      \end{equation}
      such that $t$ is $\mathrm{Tm}$-marked into
      \begin{equation}
          \begin{tikzcd}
            \cdot \arrow[r] & t.
          \end{tikzcd}
      \end{equation}
      $\mathrm{tm}_2$ corresponds to the universal property of terminal objects.
    \item
      \label{itm:pbs-commute}
      The morphism $\mathrm{pb}_0$ given by the quotient map from the free-standing non-commutative and $\mathrm{Pb}$-marked square
      \begin{equation}
        \begin{tikzcd}
          \cdot \arrow[r, "p_2"] \arrow[d, "p_1"] & \cdot \arrow[d, "f_2"] \\
          \cdot \arrow[r, "f_1"] & \cdot
        \end{tikzcd}
      \end{equation}
      to the commuting square
      \begin{equation}
        \begin{tikzcd}
          \cdot \arrow[dr, phantom, "\circlearrowleft"] \arrow[r, "p_2"] \arrow[d, "p_1"] & \cdot \arrow[d, "f_2"] \\
          \cdot \arrow[r, "f_1"] & \cdot
        \end{tikzcd}
      \end{equation}
      (which is still marked via $\mathrm{Pb}$).
      $\mathrm{pb}_0$ corresponds to the commutativity of pullback squares.
    \item
      \label{itm:pbs-exist}
      The morphism $\mathrm{pb}_1$ given by the inclusion of the cospan
      \begin{equation}
        \begin{tikzcd}
            & \cdot \arrow[d, "f_2"] \\
            \cdot \arrow[r, "f_1"] & \cdot 
        \end{tikzcd}
      \end{equation}
      with no markings into the non-commutative square
      \begin{equation}
        \begin{tikzcd}
          \cdot \arrow[r, "p_2"] \arrow[d, "p_1"] & \cdot \arrow[d, "f_2"] \\
          \cdot \arrow[r, "f_1"] & \cdot
        \end{tikzcd}
      \end{equation}
      which is marked via $\mathrm{Pb}$.
      $\mathrm{pb}_1$ corresponds to the essentially unique existence of pullback squares.
    \item
      \label{itm:pbs-factorizations}
      The morphism $\mathrm{pb}_2$ given by the inclusion of
      \begin{equation}
        \begin{tikzcd}
          \cdot \arrow[ddr, bend right, "q_1"'] \arrow[drr, bend left, "q_2"] \arrow[dr, phantom, "\circlearrowleft", near end] \\
          & \cdot \arrow[d, "p_1"] \arrow[r, "p_2"] & \cdot \arrow[d, "f_2"] \\
          & \cdot \arrow[r, "f_1"] & \cdot
        \end{tikzcd}
      \end{equation}
      in which the lower right square is non-commutative and marked via $\mathrm{Pb}$, into the diagram
      \begin{equation}
        \begin{tikzcd}
          \cdot \arrow[ddr, bend right, "q_1"'] \arrow[ddr, phantom, "\circlearrowleft"] \arrow[drr, bend left, "q_2"] \arrow[drr, phantom, "\circlearrowleft"] \arrow[dr] \\
          & \cdot \arrow[d, "p_1"] \arrow[r, "p_2"] & \cdot \arrow[d, "f_2"] \\
          & \cdot \arrow[r, "f_1"] & \cdot
        \end{tikzcd}
      \end{equation}
      in which the indicated triangles commute.
      $\mathrm{pb}_2$ corresponds to the universal property of pullback squares.
    \item
      The morphism $\mathrm{pi}_0$ given by the quotient map of the non-commutative diagram
      
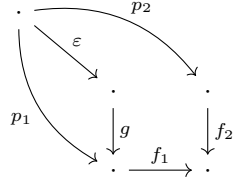
\begin{equation}
        \begin{tikzcd}
          \cdot \arrow[ddr, bend right, "p_1"'] \arrow[drr, bend left, "p_2"] \arrow[dr, "\varepsilon"] \\
          & \cdot \arrow[d, "g"] & \cdot \arrow[d, "f_2"] \\
          & \cdot \arrow[r, "f_1"] & \cdot
        \end{tikzcd}
      \end{equation}
      in which the square made of the $p_i$ and $f_i$ is marked via $\mathrm{Pb}$ and the whole diagrams is marked via $\mathrm{Pi}$, to
      
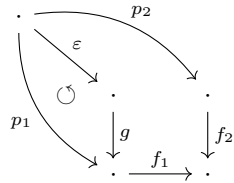
\begin{equation}
        \begin{tikzcd}
          \cdot \arrow[ddr, bend right, "p_1"'] \arrow[ddr, phantom, "\circlearrowleft"] \arrow[drr, bend left, "p_2"] \arrow[dr, "\varepsilon"] \\
          & \cdot \arrow[d, "g"] & \cdot \arrow[d, "f_2"] \\
          & \cdot \arrow[r, "f_1"] & \cdot
        \end{tikzcd}
      \end{equation}
      in which the indicated triangle commutes.
      $\mathrm{pi}_0$ corresponds to the requirement that the evaluation map $\varepsilon$ of the dependent product $f_2 = \Pi_{f_1}(g)$ is a morphism in the slice category over $\operatorname{cod} g$.
    \item
      The morphism $\mathrm{pi}_1$ given by the inclusion of a composable pair of morphisms
      
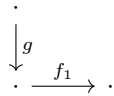
\begin{equation}
        \begin{tikzcd}
          & \cdot \arrow[d, "g"] \\
          & \cdot \arrow[r, "f_1"] & \cdot
        \end{tikzcd}
      \end{equation}
      into the non-commutative diagram
      
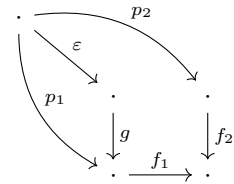
\begin{equation}
        \begin{tikzcd}
          \cdot \arrow[ddr, bend right, "p_1"] \arrow[drr, bend left, "p_2"] \arrow[dr, "\varepsilon"] \\
          & \cdot \arrow[d, "g"] & \cdot \arrow[d, "f_2"] \\
          & \cdot \arrow[r, "f_1"] & \cdot
        \end{tikzcd}
      \end{equation}
      which is marked via $\mathrm{Pi}$ (and hence the outer square is marked via $\mathrm{Pb}$).
      $\mathrm{pi}_1$ corresponds to the essentially unique existence of dependent products $f_2 = \Pi_{f_1}(g)$ and their evaluation maps $\varepsilon$.
    \item
      The morphism $\mathrm{pi}_2$ given by the inclusion of the diagram
      
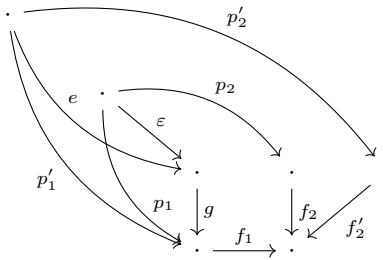
\begin{equation}
        \begin{tikzcd}
          \cdot \arrow[dddrr, bend right, "p'_1"'] \arrow[ddrrrr, bend left, "p'_2"] \arrow[ddrr, bend right, "e", pos=0.4] & \\
          & \cdot \arrow[ddr, bend right, "p_1", near end] \arrow[drr, bend left, "p_2"] \arrow[dr, "\varepsilon"] \\
          & & \cdot \arrow[d, "g"] & \cdot \arrow[d, "f_2"] & \cdot \arrow[dl, "f'_2"] \\
          & & \cdot \arrow[r, "f_1"] & \cdot
        \end{tikzcd}
      \end{equation}
      in which the square given by the $f_i$ and $p_i$ is marked via $\mathrm{Pb}$, the subdiagram given by the $f_i, p_i, g$ and $\varepsilon$ is marked via $\mathrm{Pi}$, the square given by $f_1, f'_2$ and the $p'_i$ is marked via $\mathrm{Pb}$, and $e \circ g = p'_1$, into the diagram
      
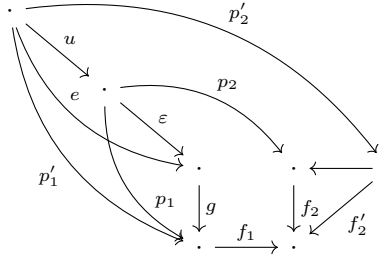
\begin{equation}
        \begin{tikzcd}
          \cdot \arrow[dr, "u"] \arrow[dddrr, bend right, "p'_1"'] \arrow[ddrrrr, bend left, "p'_2"] \arrow[ddrr, bend right, "e", pos=0.4] & \\
          & \cdot \arrow[ddr, bend right, "p_1" near end] \arrow[drr, bend left, "p_2"] \arrow[dr, "\varepsilon"] \\
          & & \cdot \arrow[d, "g"] & \cdot \arrow[d, "f_2"] & \cdot \arrow[dl, "f'_2"] \arrow[l] \\
          & & \cdot \arrow[r, "f_1"] & \cdot
        \end{tikzcd}
      \end{equation}
      in which $u$ commutes with the $p_i$ and $p'_i$, and $e = \varepsilon \circ u$.
      $\mathrm{pi}_2$ corresponds to the universal property of the dependent product $f_2 = \Pi_{f_1}(g)$.
  \end{itemize}
\end{definition}

\begin{proposition}
  \label{prop:lcc-model-cat}
  The model category $\mathrm{Lcc}$ is a model for the $(2, 1)$-category of lcc categories and lcc functors:
  \begin{enumerate}[label={(\arabic*)}]
    \item
      \label{itm:lcc-fibrant-objects}
      An object $\mathcal{C} \in \mathrm{Lcc}$ is fibrant if and only if its underlying category is lcc and
      \begin{itemize}
        \item
          a map $i(\mathrm{Tm}) \rightarrow U(\mathcal{C})$ is marked if and only if its image is a terminal object;
        \item
          a map $i(\mathrm{Pb}) \rightarrow U(\mathcal{C})$ is marked if and only if its image is a  pullback square; and
        \item
          a map $i(\mathrm{Pi}) \rightarrow U(\mathcal{C})$ is marked if and only if its image is (isomorphic to) the diagram of the evaluation map of a dependent product.
      \end{itemize}
    \item
      \label{itm:lcc-homotopy-category}
      The homotopy category of $\mathrm{Lcc}$ is equivalent to the category of lcc categories and isomorphism classes of lcc functors.
    \item
      \label{itm:lcc-homotopy-function-complexes}
      The homotopy function complexes of fibrant lcc sketches are given by the groupoids of lcc functors and their natural isomorphisms.
  \end{enumerate}
\end{proposition}
\begin{proof}
  Homotopy function complexes of maps from cofibrant to fibrant objects in a model $\mathrm{Gpd}$-category can be computed as nerves of the groupoid enrichment.
  Thus \ref{itm:lcc-homotopy-category} and \ref{itm:lcc-homotopy-function-complexes} follow from \ref{itm:lcc-fibrant-objects} and Lemma \ref{lem:simplicial-marked-objects}.

  By \citet[Theorem 4.1.1]{hirschhorn}, the fibrant objects of the left Bousfield localization $\mathrm{Lcc} = S^{-1} \mathrm{Cat}^{\mathrm{lcc}}$ at the set $S$ of morphisms from Definition \ref{def:lcc-axioms} are precisely the fibrant lcc-marked categories $\mathcal{C}$ which are $f$-local for all $f \in S$.
  The verification of the equivalence asserted in \ref{itm:lcc-fibrant-objects} can thus be split up into three parts corresponding to terminal objects, pullback squares and dependent products.
  As the three proofs are very similar, we give only the proof for pullbacks.
  For this we must show that if $\mathcal{C}$ is a $\mathrm{Pb}$-marked category, then marked maps $i(\mathrm{Pb}) \rightarrow \mathcal{C}$ are stable under isomorphisms and $\mathcal{C}$ is $\mathrm{pb}_i$-local for $i = 0, 1, 2$ if and only if the underlying category $U(\mathcal{C})$ has all pullbacks and maps $i(\mathrm{Pb}) \rightarrow U(\mathcal{C})$ are marked if and only if their images are pullbacks.

  Let $\mathcal{M}$ be a model $\mathrm{Gpd}$-category.
  The homotopy function complexes of maps from cofibrant to fibrant objects in $\mathcal{M}$ can be computed as nerves of mapping groupoids.
  The nerve functor $N : \mathrm{Gpd} \rightarrow \mathrm{sSet}$ preserves and reflects trivial fibrations.
  Thus if $f : A \rightarrow B$ is a morphism of cofibrant objects $A, B$, then a fibrant object $X$ is $f$-local if and only if
  \begin{equation}
    \label{eq:locality}
    \mathcal{M}(f, X) : \mathcal{M}(B, X) \rightarrow \mathcal{M}(A, X)
  \end{equation}
  is a trivial fibration of groupoids, i.e. an equivalence that is surjective on objects.
  
  Unfolding this we obtain the following characterization of $\mathrm{pb}_i$-locality for a fibrant $\mathrm{Pb}$-marked category:
  \begin{itemize}
    \item
      $\mathcal{C}$ is $\mathrm{pb}_0$-local if and only if all $\mathrm{Pb}$-marked squares commute.
    \item
      $\mathcal{C}$ is $\mathrm{pb}_1$-local if and only if every cospan can be completed to a $\mathrm{Pb}$-marked square, and isomorphisms of cospans can be lifted uniquely to isomorphisms of $\mathrm{Pb}$-marked squares completing them.
    \item
      $\mathcal{C}$ is $\mathrm{pb}_2$-local if and only if every commutative square completing the lower cospan of a $\mathrm{Pb}$-marked square factors via the $\mathrm{Pb}$-marked square, and every factorization is compatible with natural isomorphisms of diagrams.
      By compatibility with the identity isomorphism, the factorization is unique.
  \end{itemize}
  If these conditions are satisfied, then every cospan in $\mathcal{C}$ can be completed to a pullback square which is $\mathrm{Pb}$-marked, and $\mathrm{Pb}$-marked squares are pullbacks.
  By fibrancy of $\mathcal{C}$, it follows that precisely the pullback squares are $\mathrm{Pb}$-marked.

  Conversely, if we take as $\mathrm{Pb}$-marked squares the pullbacks in a category $\mathcal{C}$ with all pullbacks, then $\mathrm{Pb}$-marked squares will be stable under isomorphism, and, by the characterization above, $\mathcal{C}$ will be $\mathrm{pb}_i$-local for all $i$.
\end{proof}

\section{Strict lcc categories}
\label{sec:slcc}

A naive interpretation of type theory in the fibrant objects of $\mathrm{Lcc}$ as outlined in the introduction suffers from very similar issues as Seely's original version:
Type theoretic structure is preserved up to equality by substitution, but lcc functors preserve the corresponding objects with universal properties only up to isomorphism.

In this section, we explore an alternative model categorical presentation of the higher category of lcc categories.
Our goal is to rectify the deficiency that lcc functors do not preserve lcc structure up to equality.
Indeed, lcc structure on fibrant lcc sketches is induced by a right lifting \emph{property}, so there is no preferred choice of lcc structure on fibrant lcc sketches. 
We can thus not even state the required preservation up to equality.
To be able to speak of distinguished choice of lcc structure, we employ the following technical device.
\begin{definition}[\citet{algebraic-models}]
  \label{def:algebraically-fibrant-objects}
  Let $\mathcal{M}$ be a combinatorial model category and let $J$ be a set of trivial cofibrations such that objects with the right lifting property against $J$ are fibrant.
  An \emph{algebraically fibrant object} of $\mathcal{M}$ (with respect to $J$) consists of an object $G(X) \in \operatorname{Ob} \mathcal{M}$ equipped with a choice of lifts against all morphisms $j \in J$.
  Thus $X$ comes with maps $\ell_X({j, a}) : B \rightarrow G(X)$ for all $j : A \rightarrow B$ in $J$ and $a : A \rightarrow G(X)$ in $\mathcal{M}$ such that
  \begin{equation}
    \begin{tikzcd}
      A \arrow[d, "j"] \arrow[r, "a"] & G(X) \\
      B \arrow[ur, "{\ell_X(j, a)}"']
    \end{tikzcd}
  \end{equation}
  commutes.
  A morphism of algebraically fibrant objects $f : X \rightarrow Y$ is a morphism $f : G(X) \rightarrow G(Y)$ in $\mathcal{M}$ that preserves the choices of lifts, in the sense that $f \circ \ell_X(j, a) = \ell_Y(j, fa)$ for all $j : A \rightarrow B$ in $J$ and $a : A \rightarrow G(X)$.
  The category of algebraically fibrant objects is denoted by $\operatorname{Alg} \mathcal{M}$, and the evident forgetful functor $\operatorname{Alg} \mathcal{M} \rightarrow \mathcal{M}$ by $G$.
\end{definition}

\begin{proposition}
  \label{prop:object-generating-triv-cof-marked}
  Denote by $\mathcal{I} \in \operatorname{Ob} \mathrm{Gpd}$ the free-standing isomorphism with objects $0$ and $1$ and let $K \in \operatorname{Ob} I_\mathrm{lcc}$.
  Let $A_K, B_K$ be the lcc-marked object given by $U(A_K) = U(B_K) = \mathcal{I} \times K$ with $K \cong \{ 0 \} \times K \hookrightarrow \mathcal{I} \times K$ the only marking for $A_K$ and $K \cong \{ \varepsilon \} \times K \hookrightarrow \mathcal{I} \times K$, $\varepsilon = 0, 1$ the markings for $B_K$, and denote by $j_K : A_K \rightarrow B_K$ the obvious inclusion.

  Then $j_K$ is a trivial cofibration in $\mathrm{Cat}^{\mathrm{lcc}}$, and an object of $\mathrm{Cat}^{\mathrm{lcc}}$ is fibrant if and only if it has the right lifting property against $j_K$ for all $K$.
\end{proposition}
\begin{proof}
  The maps $j_K$ are injective on objects and hence cofibrations, and they reflect markings up to isomorphism, hence are also weak equivalences.
  A map $a : A_K \rightarrow \mathcal{C}$ corresponds to an isomorphism of maps $a_0, a_1 : i(K) \rightarrow \mathcal{C}$ with $a_0$ marked, and $a$ can be lifted to $B_K$ if and only if $a_1$ is also marked.
  Thus $\mathcal{C}$ has the right lifting property against the $j_K$ if and only if its markings are stable under isomorphism, which is the case if and only if $\mathcal{C}$ is fibrant.
\end{proof}

\begin{proposition}
  \label{prop:object-generating-triv-cof-lcc}
  An object of $\mathrm{Lcc}$ is fibrant if and only if it has the right lifting property against all of the following morphisms, all of which are trivial cofibrations in $\mathrm{Lcc}$:
  \begin{enumerate}[label={(\arabic*)}]
    \item
      \label{itm:marking-stability}
      The maps $j_K$ of Proposition \ref{prop:object-generating-triv-cof-marked}.
    \item
      \label{itm:lcc-axioms}
      The morphisms of Definition \ref{def:lcc-axioms}.
    \item
      \label{itm:self-unions}
      The maps $\langle \mathrm{id}, \mathrm{id} \rangle : B \amalg_A B \rightarrow B$, where $A \rightarrow B$  is one of $\mathrm{tm}_2, \mathrm{pb}_2$ or $\mathrm{pi}_2$.
  \end{enumerate}
\end{proposition}
\begin{proof}
  All three types of maps are injective on objects and hence cofibrations in $\mathrm{Cat}^{\mathrm{lcc}}$ and $\mathrm{Lcc}$.
  By Proposition \ref{prop:object-generating-triv-cof-marked}, the maps $j_K$ are trivial cofibrations of lcc-marked categories and hence also trivial cofibrations in $\mathrm{Lcc}$.

  By Proposition \ref{prop:lcc-model-cat}, the fibrant objects of $\mathrm{Lcc}$ are precisely the lcc categories. 
  If $\mathcal{C}$ is an lcc category and $f : X \rightarrow Y$ is a morphism of type \ref{itm:lcc-axioms} or \ref{itm:self-unions}, then 
  \begin{equation}
    \mathrm{Lcc}(f, \mathcal{C}) : \mathrm{Lcc}(Y, \mathcal{C}) \rightarrow \mathrm{Lcc}(X, \mathcal{C})
  \end{equation}
  is an equivalence of groupoids and hence induces a bijection of isomorphism classes.
  It follows by the Yoneda lemma that $\gamma(f)$ is an isomorphism in $\operatorname{Ho} \mathrm{Lcc}$, so $f$ is  a weak equivalence in $\mathrm{Lcc}$.

  On the other hand, let $\mathcal{C}$ be a fibrant lcc-marked category with the right lifting property against morphisms of type \ref{itm:lcc-axioms} and the morphisms of type \ref{itm:self-unions}.
  The right lifting property against $\mathrm{pb}_0, \mathrm{pb}_1$ and $\mathrm{pb}_2$ implies that $\mathrm{Pb}$-marked diagrams commute, that every cospan can be completed to a $\mathrm{Pb}$-marked square, and that every square over a cospan factors via every $\mathrm{Pb}$-marked square over the cospan.
  Uniqueness of factorizations follows from the right lifting property against the map of type \ref{itm:self-unions} corresponding to pullbacks.
  Thus $\mathcal{C}$ has pullbacks, and the argument for terminal objects and dependent products is similar.
\end{proof}

\begin{definition}
  A \emph{strict lcc category} is an algebraically fibrant object of $\mathrm{Lcc}$ with respect to the set $J$ consisting of the morphisms of types \ref{itm:marking-stability} -- \ref{itm:self-unions} of Proposition \ref{prop:object-generating-triv-cof-lcc}.
  The category of strict lcc categories is denoted by $\mathrm{sLcc}$.
\end{definition}

\begin{remark}
  The objects in the image of the forgetful functor $G : \mathrm{sLcc} \rightarrow \mathrm{Lcc}$ are the fibrant lcc sketches, i.e.\@ lcc categories.
  To endow an lcc category with the structure of a strict lcc category, we need to choose canonical lifts $\ell(j, -)$ against the morphisms $j \in J$.
  Because the lifts against all other morphisms are uniquely determined, only the choices for $\mathrm{tm}_1, \mathrm{pb}_1$ and $\mathrm{pi}_1$ are relevant for this.
  Thus a strict lcc category is an lcc category with assigned terminal object, pullback squares and dependent products (including the evaluation maps of dependent products).
  A strict lcc functor is then an lcc functor that preserves these canonical choices of universal objects not just up to isomorphism but up to equality.
\end{remark}

The slice category $\mathcal{C}_{/ \sigma}$ over an object $\sigma$ of an lcc category $\mathcal{C}$ is lcc again.
A morphism $s : \sigma \rightarrow \tau$ in $\mathcal{C}$ induces by pullback an lcc functor $s^* : \mathcal{C}_{/ \tau} \rightarrow \mathcal{C}_{/ \sigma}$, and there exist functors $\Pi_s, \Sigma_s : U(\mathcal{C}_{/ \tau}) \rightarrow U(\mathcal{C}_{/ \sigma})$ and adjunctions $\Sigma_s \dashv U(s^*) \dashv \Pi_s$.
These data depend on choices of pullback squares and dependent products, and hence they are preserved by lcc functors only up to isomorphism.

For strict lcc categories $\Gamma$, however, these functors can be constructed using canonical lcc structure, i.e.\@ using the lifts $\ell(j, -)$ for various $j \in J$, and this choice is preserved by strict lcc functors.
\begin{proposition}
  \label{prop:strict-slicing}
  Let $\Gamma$ be a strict lcc category, and let $\sigma \in \operatorname{Ob} \Gamma$.
  Then there is a strict lcc category $\Gamma_{/ \sigma}$ whose underlying category is the slice $U(G(\Gamma))_{/ \sigma}$.

  If $s : \sigma \rightarrow \tau$ is a morphism in $\Gamma$, then there is a canonical choice of pullback functor $s^* : G(\Gamma_{/ \tau}) \rightarrow G(\Gamma_{/ \sigma})$ which is lcc (but not necessarily strict) and canonical left and right adjoints
  \begin{equation}
    \Sigma_s \dashv U(s^*) \dashv \Pi_s.
  \end{equation}

  These data are natural in $\Gamma$.
  Thus if $f : \Gamma \rightarrow \Delta$ is strict lcc, then the evident functor $f_{/ \sigma} : U(G(\Gamma_{/ \sigma})) \rightarrow U(G(\Delta_{ / f(\sigma)}))$ is strict lcc, and the following squares in $\mathrm{Lcc}$ respectively $\mathrm{Cat}$ commute:
  \begin{mathpar}
    \begin{tikzcd}[sep=large]
      \Gamma_{/ \sigma} \arrow[r, "f_{/\sigma}"] & \Delta_{/ f(\sigma)} \\
      \Gamma_{/ \tau} \arrow[r, "f_{/ \tau}"] \arrow[u, "s^*"'] & \Delta_{/ f(\tau)} \arrow[u, "f(s)^*"']
    \end{tikzcd}
    \and
    \begin{tikzcd}[sep=large]
      \Gamma_{/ \sigma} \arrow[r, "f_{/\sigma}"] \arrow[d, "\Sigma_s"] & \Delta_{/ f(\sigma)} \arrow[d, "\Sigma_{f(s)}"] \\
      \Gamma_{/ \tau} \arrow[r, "f_{/ \tau}"] & \Delta_{/ f(\tau)}
    \end{tikzcd}
    \and
    \begin{tikzcd}[sep=large]
      \Gamma_{/ \sigma} \arrow[r, "f_{/\sigma}"] \arrow[d, "\Pi_s"] & \Delta_{/ f(\sigma)} \arrow[d, "\Pi_{f(s)}"] \\
      \Gamma_{/ \tau} \arrow[r, "f_{/ \tau}"] & \Delta_{/ f(\tau)}
    \end{tikzcd}
  \end{mathpar}
  (Here application of $G$ and $U$ has been omitted; the left square is valued in $\mathrm{Lcc}$, and the two right squares are valued in $\mathrm{Cat}$.)
  $f_{/ \sigma}$ and $f_{/ \tau}$ commute with taking transposes along the involved adjunctions.
\end{proposition}
\begin{proof}
  We take as canonical terminal object of $\Gamma_{/ \sigma}$ the identity morphism $\mathrm{id}_\sigma$ on $\sigma$.
  Canonical pullbacks in $\Gamma_{/ \sigma}$ are computed as canonical pullbacks of the underlying diagram in $\Gamma$, and similarly for dependent products. 
  
  The canonical pullback and dependent product functors $\sigma^*, \Pi_s$ are defined using canonical pullbacks and dependent products, and dependent sum functors $\Sigma_s$ are computed by composition with $s$.
  Units and counits of the adjunctions are given by the evaluation maps of canonical dependent products and the projections of canonical pullbacks.

  Because these data are defined in terms of canonical lcc structure on $\Gamma$, they are preserved by strict lcc functors.
\end{proof}

The context morphisms in our categories with families (cwfs) \citep{internal-type-theory} will usually be defined as functors of categories in the opposite directions.
Cwfs are categories equipped with contravariant functors to $\mathrm{Fam}$, the category of families of sets.
To avoid having to dualize twice, we thus introduce the following notion.
\begin{definition}
  A \emph{covariant cwf} is a category $\mathcal{C}$ equipped with a (covariant) functor $(\mathrm{Ty}, \mathrm{Tm}) : \mathcal{C} \rightarrow \mathrm{Fam}$.
\end{definition}
The intuition for a context morphism $f : \Gamma \rightarrow \Delta$ in a cwf is an assignment of terms in $\Gamma$ to the variables occurring in $\Delta$.
Dually, a morphism $f : \Delta \rightarrow \Gamma$ in a covariant cwf should thus be thought of as a mapping of the variables in $\Delta$ to terms in context $\Gamma$, or more conceptually as an interpretation of the mathematics internal to $\Delta$ into the mathematics internal to $\Gamma$.

Apart from our use of covariant cwfs, we adhere to standard terminology with the obvious dualization.
For example, an empty context in a covariant cwf is an initial (instead of terminal) object in the underlying category.

To distinguish type and term formers in (covariant) cwfs from the corresponding categorical structure, the type theoretic notions are typeset in bold where confusion is possible.
Thus $\mathbf{\Pi}_\sigma \, \tau$ denotes a dependent product type whereas $\Pi_\sigma(\tau)$ denotes application of a dependent product functor $\Pi_\sigma : \mathcal{C}_{/ \sigma} \rightarrow \mathcal{C}$ to an object $\tau \in \operatorname{Ob} \mathcal{C}_{/ \sigma}$.

\begin{definition}
  The covariant cwf structure on $\mathrm{sLcc}$ is given by $\mathrm{Ty}(\Gamma) = \operatorname{Ob} \Gamma$ and $\mathrm{Tm}(\Gamma, \sigma) = \mathrm{Hom}_\Gamma(1, \sigma)$, where $1$ denotes the canonical terminal object of $\Gamma$.
\end{definition}

\begin{proposition}
  \label{prop:slcc-model}
  The covariant cwf $\mathrm{sLcc}$ has an empty context and context extensions, and it supports finite product and extensional equality types.
\end{proposition}
\begin{proof}
  It follows from Theorem \ref{th:algebraically-fibrant-model-category} below that $\mathrm{sLcc}$ is cocomplete and that $G : \mathrm{sLcc} \rightarrow \mathrm{Lcc}$ has a left adjoint $F$.
  In particular, there exists an initial strict lcc category, i.e.\@ an empty context in $\mathrm{sLcc}$.
  
  Let $\Gamma \vdash \sigma$.
  The context extension $\Gamma.\sigma$ is constructed as pushout
  \begin{equation}
    \label{eq:slcc-context-extension}
    \begin{tikzcd}
      F(\{ t, \sigma \}) \arrow[d] \arrow[r] & F(\{ v : t \rightarrow \sigma \}) \arrow[d] \\
      \Gamma \arrow[r, "p"] & \Gamma.\sigma
    \end{tikzcd}
  \end{equation}
  where $\{ t, \sigma \}$ denotes a minimally marked lcc sketch with two objects and $\{ v : t \rightarrow \sigma \}$ is the minimally marked free-standing arrow.
  The vertical morphism on the left is induced by mapping $t$ to $1$ (the canonical terminal object of $\Gamma$) and $\sigma$ to $\sigma$, and the top morphism is the evident inclusion.
  The variable $\Gamma.\sigma \vdash v : p(\sigma)$ is given by the image of $v$ in $\Gamma.\sigma$.

  Unit types $\Gamma \vdash 1$ are given by the canonical terminal objects of strict lcc categories $\Gamma$.
  Binary product types $\Gamma \vdash \sigma \times \tau$ are given by canonical pullbacks $\sigma \times_1 \tau$ over the canonical terminal object $1$ in $\Gamma$.
  Finally, equality types $\Gamma \vdash \mathrm{Eq} s \, t$ are constructed as canonical pullbacks
  \begin{equation}
    \begin{tikzcd}
      \operatorname{\mathrm{Eq}} s \, t \arrow[r] \arrow[d] & 1 \arrow[d, "t"] \\
      1 \arrow[r, "s"] & \sigma,
    \end{tikzcd}
  \end{equation}
  in $\Gamma$, i.e.\@ as equalizers of $s$ and $t$.

  Because these type constructors (and evident term formers) are defined from canonical lcc structure, they are stable under substitution.
\end{proof}

\begin{remark}
  \label{rem:slcc-slice-problem}
  Unfortunately, $\mathrm{sLcc}$ does not support dependent product or dependent sum types in a similarly obvious way.
  The introduction rule for dependent types is
  \begin{equation}
    \inferrule
    {\Gamma \vdash \sigma \\ \Gamma.\sigma \vdash \tau}
    {\Gamma \vdash \mathbf{\Pi}_\sigma \tau}.
  \end{equation}
  To interpret it, we would like to apply the dependent product functor $\Pi_\sigma : U(G(\Gamma))_{/ \sigma} \rightarrow U(G(\Gamma))$ to $\tau$.
  
  We thus need a functor $U(G(\Gamma.\sigma)) \rightarrow U(G(\Gamma_{/ \sigma}))$ to obtain an object of the slice category, and the construction of such a functor appears to be not generally possible.
  Note that the most natural strategy for constructing this functor using the universal property of $\Gamma.\sigma$ does not work:
  For this we would note that the pullback functor $\sigma^* : G(\Gamma) \rightarrow G(\Gamma_{/ \sigma})$ is lcc, and that the diagonal $d : \sigma \rightarrow \sigma \times \sigma$ is a morphism $1 \rightarrow \sigma^*(\sigma))$ in $\Gamma_{/ \sigma}$, and then try to obtain $\langle \sigma^*, d \rangle : \Gamma.\sigma \rightarrow \Gamma_{/ \sigma}$.
  The flaw in this argument is that $\sigma^*$, while lcc, is not strict, and the universal property of $\Gamma.\sigma$ only applies to strict lcc functors.
  A solution to this problem is presented in Section \ref{sec:algebraically-cofibrant}.
\end{remark}

We conclude the section with a justification for why we have not gone astray so far:
The initial claim was that our interpretation of type theory would be valued in the category of lcc categories, but $\mathrm{sLcc}$ is neither 1-categorically nor bicategorically equivalent to the category $\mathrm{Lcc}_f$ of fibrant lcc sketches.
Indeed, not every non-strict lcc functor of strict lcc categories is isomorphic to a strict lcc functor.
Nevertheless, $\mathrm{sLcc}$ has model category structure that presents the same higher category of lcc categories by the following theorem:
\begin{theorem}[\citet{algebraic-models} Proposition 2.4, \citet{equipping-weak-equivalences} Theorem 19]
  \label{th:algebraically-fibrant-model-category}
  Let $\mathcal{M}$ be a combinatorial model category, and let $J$ be a set of trivial cofibrations such that objects with the right lifting property against $J$ are fibrant.
  Then $G : \operatorname{Alg} \mathcal{M} \rightarrow \mathcal{M}$ is monadic with left adjoint $F$, and $\operatorname{Alg} \mathcal{M}$ is a locally presentable category.
  The model structure of $\mathcal{M}$ can be transferred along the adjunction $F \dashv G$ to $\operatorname{Alg} \mathcal{M}$, endowing $\operatorname{Alg} \mathcal{M}$ with the structure of a combinatorial model category.
  $F \dashv G$ is a Quillen equivalence, and the unit $X \rightarrow G(F(X))$ is a trivial cofibration for all $X \in \mathcal{M}$.
\end{theorem}
Theorem \ref{th:algebraically-fibrant-model-category} appears in \citet{algebraic-models} with the additional assumption that all cofibrations in $\mathcal{M}$ are monomorphisms.
This assumption is lifted in \citet{equipping-weak-equivalences}, but there $J$ is a set of generating trivial cofibrations, which is a slightly stronger condition than the one stated in the theorem.
However, the proof in \citet{equipping-weak-equivalences} works without change in the more general setting.

That the model structure of $\operatorname{Alg} \mathcal{M}$ is obtained by transfer from that of $\mathcal{M}$ means that $G$ reflects fibrations and weak equivalences.

\begin{lemma}
  \label{lem:algebraically-fibrant-lproper-simplicial}
  Let $\mathcal{M}$ and $J$ be as in Theorem \ref{th:algebraically-fibrant-model-category}, and suppose furthermore that $\mathcal{M}$ is a model $\mathrm{Gpd}$-category.
  Then $\operatorname{Alg} \mathcal{M}$ has the structure of a model $\mathrm{Gpd}$-category, and the adjunction $F \dashv G$ lifts to a Quillen $\mathrm{Gpd}$-adjunction.
\end{lemma}
\begin{proof}
  Let $X$ and $Y$ be algebraically fibrant objects.
  We define the mapping groupoid $(\operatorname{Alg} \mathcal{M})(X, Y)$ to be the full subgroupoid of $\mathcal{M}(G(X), G(Y))$ whose objects are the maps of algebraically fibrant objects $X \rightarrow Y$.

  Because $\mathrm{Gpd}$ is generated under colimits by the free-standing isomorphism $\mathcal{I}$, it will follow from the existence of powers $X^\mathcal{I}$ that $\operatorname{Alg} \mathcal{M}$ is complete as a $\mathrm{Gpd}$-category.
  As we will later show that $G$ is a right adjoint, the powers in $\operatorname{Alg} \mathcal{M}$ must be constructed such that they commute with $G$, i.e.\@ $G(X^\mathcal{I}) = G(X)^\mathcal{I}$.

  Let $j : A \rightarrow B$ be in $J$ and let $a : A \rightarrow G(X)^\mathcal{I}$.
  The canonical lift $\ell(j, a) : B \rightarrow G(X)^\mathcal{I}$ is constructed as follows:
  $a$ corresponds to a map $\bar a : \mathcal{I} \rightarrow \mathcal{M}(A, G(X))$, i.e.\@ an isomorpism of maps $A \rightarrow G(X)$.
  Its source and target are morphisms $\bar a_0, \bar a_1 : A \rightarrow G(X)$, for which we obtain canonical lifts $\ell(j, \bar a_i) : B \rightarrow G(X)$ using the canonical lifts of $X$.
  Because $G(X)$ is fibrant and $j$ is a trivial cofibration, the map $\mathcal{M}(j, G(X)) : \mathcal{M}(B, G(X)) \rightarrow \mathcal{M}(A, G(X))$ is a trivial fibration and in particular an equivalence.
  It follows that $\bar a$ can be lifted uniquely to an isomorphism of $\ell(j, \bar a_0)$ with $\ell(j, \bar a_1)$, and we take $\ell(j, a) : B \rightarrow G(X)^\mathcal{I}$ as this isomorphism's transpose.

  From uniqueness of the lift defining $\ell(j, a)$, it follows that a map $G(Y) \rightarrow G(X)^\mathcal{I}$ preserves canonical lifts if and only if the two maps
  \begin{equation}
    \begin{tikzcd}[column sep=large]
      G(Y) \arrow[r] & G(X)^\mathcal{I} \arrow[r, "G(X)^{\{i\}}"] & G(X)
    \end{tikzcd}
  \end{equation}
  given by evaluation at the endpoints $i = 0, 1$ of the isomorphism $\mathcal{I}$ preserve canonical lifts.
  Thus the canonical isomorphism
  \begin{equation}
    \mathrm{Gpd}(\mathcal{I}, \mathcal{M}(G(Y), G(X))) \cong \mathcal{M}(G(Y), G(X)^\mathcal{I})
  \end{equation}
  restricts to an isomorphism
  \begin{equation}
    \mathrm{Gpd}(\mathcal{I}, (\operatorname{Alg} \mathcal{M})(Y, X)) \cong (\operatorname{Alg} \mathcal{M})(Y, X^\mathcal{I}).
  \end{equation}

  It follows by \citet[theorem 4.85]{basic-concepts-of-enriched-category-theory} and the preservation of powers by $G$ that the 1-categorical adjunction $F \dashv G$ is groupoid enriched.
  It is proved in \citet{algebraic-models} that $G$, when considered as a functor of ordinary categories, is monadic using Beck's monadicity theorem.
  The only additional assumption for the enriched version of Beck's theorem \citep[theorem II.2.1]{enriched-kan-extensions} we have to check is that the coequalizer of a $G$-split pair of morphisms as constructed in \citet{algebraic-models} is a colimit also in the enriched sense.
  This follows immediately from the fact that $G$ is locally full and faithful.
  $G$ is $\mathrm{Gpd}$-monadic and accessible, so $\operatorname{Alg} \mathcal{M}$ is $\mathrm{Gpd}$-cocomplete by \citet[theorem 3.8]{two-dimensional-monad-theory}.

  It remains to show that $\operatorname{Alg} \mathcal{M}$ is groupoid enriched also in the model categorical sense.
  For this it suffices to note that $G$ preserves (weighted) limits and that $G$ preserves and reflects fibrations and weak equivalences, so that the map
  \begin{equation}
    X^\mathcal{H} \rightarrow X^\mathcal{G} \times_{Y^\mathcal{G}} Y^\mathcal{H}
  \end{equation}
  induced by a cofibration of groupoids $f : \mathcal{G} \rightarrowtail \mathcal{H}$ and a fibration $g : X \rightarrow Y$ in $\operatorname{Alg} \mathcal{M}$ is a fibration and a weak equivalence if either $f$ or $g$ is a weak equivalence.
\end{proof}

\begin{remark}
  \citet{homotopy-theoretic-aspects} defines model category structure on $T\text{-Alg}_s$, the category of strict algebras and their strict morphisms for a 2-monad $T$ on a model $\mathrm{Gpd}$-category.
  If we choose for $T$ the monad on $\mathrm{Cat}$ assigning to every category the free lcc category generated by it, then $T\text{-Alg}_s$ is $\mathrm{Gpd}$-equivalent to $\mathrm{sLcc}$, so it is natural ask whether their model category structures agree.

  The model category structure on $T\text{-Alg}_s$ is defined by transfer from $\mathrm{Cat}$, i.e. such that the forgetful functor $T\text{-Alg}_s \rightarrow \mathrm{Cat}$ reflects fibrations and weak equivalences.
  The same is true for $\mathrm{sLcc} \rightarrow \mathrm{Lcc}$, and this functor is valued in the fibrant objects of $\mathrm{Lcc}$.
  The restriction of the functor $\mathrm{Lcc} \rightarrow \mathrm{Cat}$ to fibrant objects reflects weak equivalences and trivial fibrations because equivalences of categories preserve and reflect universal objects that exist in domain and codomain.
  Thus $\mathrm{sLcc}$ and $T\text{-Alg}_s$ have the same sets of weak equivalences and trivial fibrations, hence their model category structures coincide.
\end{remark}

\section{Algebraically cofibrant strict lcc categories}
\label{sec:algebraically-cofibrant}

As noted in Remark \ref{rem:slcc-slice-problem}, to interpret dependent sum and dependent product types in $\mathrm{sLcc}$, we would need to relate context extensions $\Gamma.\sigma$ to slice categories $\Gamma_{/ \sigma}$.
In this section we discuss how this problem can be circumvented by considering yet another Quillen equivalent model category: The category of algebraically \emph{cofibrant} strict lcc categories.

The slice category $\mathcal{C}_{/ x}$ of an lcc category $\mathcal{C}$ is bifreely generated by (any choice of) the pullback functor $\sigma^* : \mathcal{C} \rightarrow \mathcal{C}_{/ x}$ and the diagonal $d : x \rightarrow x \times x$, viewed as a morphism $1 \rightarrow x^*(x)$ in $\mathcal{C}_{/ x}$:
Given a pair of lcc functor $f : \mathcal{C} \rightarrow \mathcal{D}$ and morphism $s : 1 \rightarrow f(x)$ in $\mathcal{D}$, there is an lcc functor $g : \mathcal{C}_{/ x} \rightarrow \mathcal{D}$ that commutes with $f$ and $x^*$ up to a natural isomorphism under which $g(d)$ corresponds to $s$, and every other lcc functor with this property is uniquely isomorphic to $g$.

Phrased in terms of model category theory, this biuniversal property amounts to asserting that the square
\begin{equation}
  \label{eq:homotopy-pushout-lcc}
  \begin{tikzcd}
    \{ t, x \} \arrow[r] \arrow[d] & \{ d : t \rightarrow x \} \arrow[d] \\
    \mathcal{C} \arrow[r, "x^*"] & \mathcal{C}_{/ x}
  \end{tikzcd}
\end{equation}
is a homotopy pushout square in $\mathrm{Lcc}$.
Here $\{t, x\} = \{t, x\}^\flat$ denotes the discrete category with two objects and no markings, from which $\{d : t \rightarrow x\}$ is obtained by adjoining a single morphism $t \rightarrow x$.
The left vertical map $\{ t, x \} \rightarrow \mathcal{C}$ maps $t$ to some terminal object and $x$ to $x$, and the right vertical map maps $d$ to the diagonal $x \rightarrow x^*(x)$ in $\mathcal{C}_{/ x}$.

Recall from Proposition \ref{prop:slcc-model} that a context extension $\Gamma.\sigma$ in $\mathrm{sLcc}$ is defined by the 1-categorical pushout square
\begin{equation}
  \label{eq:slcc-context-extension-2}
  \begin{tikzcd}
    F(\{ t, \sigma \}) \arrow[d] \arrow[r] & F(\{ v : t \rightarrow \sigma \}) \arrow[d] \\
    \Gamma \arrow[r, "p"] & \Gamma.\sigma.
  \end{tikzcd}
\end{equation}
Because $F \dashv G$ is a Quillen equivalence, we should thus expect to find weak equivalences relating $\Gamma_{/ \sigma}$ to $\Gamma.\sigma$ if the pushout \eqref{eq:slcc-context-extension-2} is also a \emph{homotopy} pushout.

By \citet[Proposition A.2.4.4]{higher-topos-theory}, this is the case if $\Gamma$, $F(\{t, \sigma\})$ and $F(\{v : t, \sigma\})$ are cofibrant, and the map $F(\{t, \sigma\}) \rightarrow F(\{v : t, \sigma\})$ is a cofibration.
The cofibrations of $\mathrm{Lcc}$ are the maps which are injective on objects.
It follows that $\{t, \sigma\}$ and $\{ v : t \rightarrow \sigma\}$ are cofibrant lcc sketches, and that the inclusion of the former into the latter is a cofibration.
$F$ is a left Quillen functor and hence preserves cofibrations.
Thus the pushout \eqref{eq:slcc-context-extension-2} is a homotopy pushout if $\Gamma$ is cofibrant.

Note that components of the counit $\varepsilon : FG \Rightarrow \mathrm{Id} : \mathrm{sLcc} \rightarrow \mathrm{sLcc}$ are cofibrant replacements:
Every lcc sketch is cofibrant in $\mathrm{Lcc}$, every strict lcc category is fibrant in $\mathrm{sLcc}$, and $F \dashv G$ is a Quillen equivalence.
It follows that a strict lcc category $\Gamma$ is cofibrant if and only if the counit $\varepsilon_\Gamma$ is a retraction, say with section $\lambda : \Gamma \rightarrow F(G(\Gamma))$.

And indeed, this section can be used to strictify the pullback functor.
We have $\sigma^* : G(\Gamma) \rightarrow G(\Gamma_{/ \sigma})$, which induces a strict lcc functor $\overline{\sigma^*} : F(G(\Gamma)) \rightarrow \Gamma_{/ \sigma}$.
Now let
\begin{equation}
  (\sigma^*)^s : \Gamma \xrightarrow{\lambda} F(G(\Gamma)) \xrightarrow{\overline{\sigma^*}} \Gamma_{/ \sigma},
\end{equation}
which is naturally isomorphic to $\sigma^*$.
Adjusting the domain and codomain of the diagonal $d$ suitably to match $(\sigma^*)^s$, we thus obtain the desired comparison functor $\langle \lambda (\sigma^*)^s, d \rangle : \Gamma.\sigma \rightarrow \Gamma_{/ \sigma}$.

At first we might thus attempt to restrict the category of contexts to the cofibrant strict lcc categories $\Gamma$, for which sections $\lambda : \Gamma \rightarrow F(G(\Gamma))$ exist.
Indeed, cofibrant objects are stable under pushouts along cofibrations, so the context extension $\Gamma.\sigma$ will be cofibrant again if $\Gamma$ is cofibrant.
The dependent product type $\mathbf{\Pi}_\sigma \, \tau$ would be defined by application of
\begin{equation}
  \begin{tikzcd}
    \Gamma.\sigma \arrow[r] & \Gamma_{/ \sigma} \arrow[r, "\Pi_\sigma"] & \Gamma
  \end{tikzcd}
\end{equation}
to $\tau$.
Unfortunately, the definition of the comparison functor $\Gamma.\sigma \rightarrow \Gamma_{/ \sigma}$ required a \emph{choice} of section $\lambda : \Gamma \rightarrow F(G(\Gamma))$, and this choice will not generally be compatible with strict lcc functors $\Gamma \rightarrow \Delta$.
The dependent products defined as above will thus not be stable under substitution.

To solve this issue, we make the section $\lambda$ part of the structure.
Similarly to how strict lcc categories have associated structure corresponding to their fibrancy in lcc, we make the section $\lambda$ witnessing the cofibrancy of strict lcc categories part of the data, and require morphisms to preserve it.
We thus consider algebraically cofibrant objects, which, dually to algebraically fibrant objects, are defined as coalgebras for a cofibrant replacement comonad.
As in the case of algebraically fibrant objects, we are justified in doing so because we obtain an equivalent model category:

\begin{theorem}[\citet{coalgebraic-models} Lemmas 1.2 and 1.3, Theorems 1.4 and 2.5]
  \label{th:coalgebraic-model-category}
  Let $\mathcal{M}$ be a combinatorial and model $\mathrm{Gpd}$-category.
  Then there are arbitrarily large cardinals $\lambda$ such that
  \begin{enumerate}[label={(\arabic*)}]
    \item
      $\mathcal{M}$ is locally $\lambda$-presentable;
    \item
      $\mathcal{M}$ is cofibrantly generated with a set of generating cofibrations for which domains and codomains are $\lambda$-presentable objects;
    \item
      an object $X \in \operatorname{Ob} \mathcal{M}$ is $\lambda$-presentable if and only if the functor $\mathcal{M}(X, -) : \mathcal{M} \rightarrow \mathrm{Gpd}$, given by the groupoid enrichment of $\mathcal{M}$, preserves $\lambda$-filtered colimits.
  \end{enumerate}

  Let $\lambda$ be any such cardinal.
  Then there is a cofibrant replacement $\mathrm{Gpd}$-comonad $C : \mathcal{M} \rightarrow \mathcal{M}$ that preserves $\lambda$-filtered colimits. 
  Let $C$ be any such comonad and denote its category of coalgebras by $\operatorname{Coa} \mathcal{M}$.

  Then the forgetful functor $U : \operatorname{Coa} \mathcal{M} \rightarrow \mathcal{M}$ has a left adjoint $V$.
  $\operatorname{Coa} \mathcal{M}$ is a complete and cocomplete $\mathrm{Gpd}$-category, and $V \dashv U$ is a $\mathrm{Gpd}$-adjunction.
  The model category structure of $\mathcal{M}$ can be transferred along $V \dashv U$, making $\operatorname{Coa} \mathcal{M}$ a model $\mathrm{Gpd}$-category.
  $V \dashv U$ is a Quillen equivalence.
\end{theorem}

For $\mathcal{M} = \mathrm{sLcc}$, the first infinite cardinal $\omega$ satisfies the three conditions of Theorem \ref{th:coalgebraic-model-category}, and $C = FG$ is a suitable cofibrant replacement comonad.

\begin{definition}
  The covariant cwf structure on $\operatorname{Coa} \mathrm{sLcc}$ is defined as the composite
  \begin{equation}
    \operatorname{Coa} \mathrm{sLcc} \rightarrow \mathrm{sLcc} \rightarrow \mathrm{Fam}
  \end{equation}
  in terms of the covariant cwf structure on $\mathrm{sLcc}$.
\end{definition}

We denote by $\eta : \mathrm{Id} \Rightarrow G F : \mathrm{Lcc} \rightarrow \mathrm{Lcc}$ the unit and by $\varepsilon : F G \Rightarrow \mathrm{Id} : \mathrm{sLcc} \rightarrow \mathrm{sLcc}$ the counit of the adjunction $F \dashv G$.

\begin{lemma}
  \label{lem:coalgebra-vs-eta}
  Let $\lambda : \Gamma \rightarrow F(G(\Gamma))$ be an $FG$-coalgebra.
  Then there is a canonical natural isomorphism $\phi : G(\lambda) \cong \eta_{G(\Gamma)} : G(\Gamma) \rightarrow G(F(G(\Gamma)))$ of lcc functors which is compatible with morphisms of $FG$-coalgebras.
\end{lemma}
\begin{proof}
  It suffices to construct a natural isomorphism
  \begin{equation}
    \psi : \mathrm{id} \cong \eta_{G(\Gamma)} \circ G(\varepsilon_\Gamma)
  \end{equation}
  of lcc endofunctors on $G(F(G(\Gamma)))$ for every strict lcc category $\Gamma$, because then
  \begin{equation}
    \psi \circ G(\lambda) : G(\lambda) \cong \eta_{G(\Gamma)} \circ G(\varepsilon_\Gamma \lambda) = \eta_{G(\Gamma)}
  \end{equation}
  for every coalgebra $\lambda : \Gamma \rightarrow F(G(\Gamma))$.

  $\eta_{G(\Gamma)}$ is a trivial cofibration, so the map
  \begin{equation}
    \label{eq:eta-precomposition}
    - \circ \eta_{G(\Gamma)} : \mathrm{Lcc}(G(F(G(\Gamma))), G(F(G(\Gamma)))) \rightarrow \mathrm{Lcc}(G(\Gamma), G(F(G(\Gamma)))) 
  \end{equation}
  is a trivial fibration of groupoids.
  By one of the triangle identities of units and counits, we have $\eta_{G(\Gamma)} \circ G(\varepsilon_\Gamma) \circ \eta_{G(\Gamma)} = \eta_{G(\Gamma)}$.
  Thus both $\mathrm{id}_{G(\Gamma)}$ and $\eta_{G(\Gamma)} \circ G(\varepsilon_\Gamma)$ are sent to $\eta_{G(\Gamma)}$ under the surjective equivalence \eqref{eq:eta-precomposition}, and so we can lift the identity natural isomorphism on $\eta_{G(\Gamma)}$ to an isomorphism $\psi$ as above.
  Since the lift is unique, it is preserved under strict lcc functors in $\Gamma$.
\end{proof}

\begin{proposition}
  \label{prop:coa-slcc-exts}
  The covariant cwf $\operatorname{Coa} \mathrm{sLcc}$ has an empty context and context extensions, and the forgetful functor $\operatorname{Coa} \mathrm{sLcc} \rightarrow \mathrm{sLcc}$ preserves both.
\end{proposition}
\begin{proof}
  The model category $\operatorname{Coa} \mathrm{sLcc}$ has an initial object, i.e.\@ an empty context.
  Its underlying strict lcc category $\Gamma$ is the initial strict lcc category, and the structure map $\lambda : \Gamma \rightarrow F(G(\Gamma))$ is the unique strict lcc functor with this signature.

  Now let $\lambda : \Gamma \rightarrow F(G(\Gamma))$ be an $FG$-coalgebra and $\Gamma \vdash \sigma$ be a type.
  We must construct coalgebra structure $\lambda.\sigma : \Gamma.\sigma \rightarrow F(G(\Gamma.\sigma))$ on the context extension in $\mathrm{sLcc}$ such that
  \begin{equation}
    \begin{tikzcd}
      \Gamma \arrow[r, "p"] \arrow[d, "\lambda"] & \Gamma.\sigma \arrow[d, "\lambda.\sigma"] \\
      F(G(\Gamma)) \arrow[r, "F(G(p))"] & F(G(\Gamma.\sigma))
    \end{tikzcd}
  \end{equation}
  commutes, and show that the strict lcc functor $\langle f, w \rangle : \Gamma.\sigma \rightarrow \Delta$ induced by a coalgebra morphism $f : (\Gamma, \lambda) \rightarrow (\Delta, \lambda')$ and a term $\Delta \vdash w : f(\sigma)$ is a coalgebra morphism.

  Let $v : 1 \rightarrow p(\sigma)$ be the variable term of the context extension of $\Gamma$ by $\sigma$.
  Then $\eta_{\Gamma.\sigma}(v)$ is a morphism
  \begin{equation}
    \eta_{\Gamma.\sigma}(1) \rightarrow \eta_{\Gamma.\sigma}(p(\sigma)) = F(G(p))(\eta_\Gamma(\sigma))
  \end{equation}
  in $F(G(\Gamma.\sigma))$.
  $\eta_{\Gamma.\sigma}(1)$ is a terminal object and hence uniquely isomorphic to the canonical terminal object $1$ of $F(G(\Gamma.\sigma))$, and $F(G(p))(\eta_\Gamma(\sigma))$ is isomorphic to $F(G(p))(\lambda(\sigma))$ via a component of $F(G(p)) \circ \phi$, where $\phi$ is the natural isomorphism constructed in Lemma \ref{lem:coalgebra-vs-eta}.
  We thus obtain a term $\Gamma.\sigma \vdash v' : F(G(p))(\lambda(\sigma))$ and can define
  \begin{equation}
    \lambda.\sigma = \langle F(G(p)) \circ \lambda, v' \rangle
  \end{equation}
  by the universal property of $\Gamma.\sigma$.
  $\lambda.\sigma$ is compatible with $p$ and $\lambda$ by construction.

  Now let $f : (\Gamma, \lambda) \rightarrow (\Delta, \lambda')$ be a coalgebra morphism and let $\Delta \vdash w : f(\sigma)$.
  We need to show that
  \begin{equation}
    \begin{tikzcd}[column sep=large, row sep=large]
      \Gamma.\sigma \arrow[r, "{\langle f, w \rangle}"] \arrow[d, "\lambda.\sigma"] & \Delta \arrow[d, "\lambda'"] \\
      F(G(\Gamma.\sigma)) \arrow[r, "{F(G(\langle f, w \rangle}))"] & F(G(\Delta))
    \end{tikzcd}
  \end{equation}
  commutes.
  This follows from the universal property of $\Gamma.\sigma$:
  The two maps $\Gamma.\sigma \rightarrow F(G(\Delta))$ agree after precomposing $p : \Gamma \rightarrow \Gamma.\sigma$ because by assumption $f$ is a coalgebra morphism, and they both map $v$ to the term $F(G(\Delta)) \vdash w' : \lambda'(f(\sigma))$ obtained from $w$ similarly to $v'$ from $v$ because the isomorphism $\phi$ constructed in Lemma \ref{lem:coalgebra-vs-eta} is compatible with coalgebra morphisms.
\end{proof}

For $\mathcal{C}$ a $\mathrm{Gpd}$-category and $x \in \operatorname{Ob} \mathcal{C}$, we denote by $\mathcal{C}_{x / }$ the higher coslice $\mathrm{Gpd}$-category of objects under $x$.
Its objects are morphisms out of $x$, its morphisms are triangles
\begin{equation}
  \begin{tikzcd}
    & x \arrow[dl, "y_0"'] \arrow[dr, "y_1"] \arrow[d, "\overset{\phi}{\cong}", phantom] & \\
    \cdot \arrow[rr, "f"'] & \, & \cdot
  \end{tikzcd}
\end{equation}
in $\mathcal{C}$ which commute up to specified isomorphism $\phi$, and its 2-cells $(f_0, \phi_0) \cong (f_1, \phi_1)$ are 2-cells $\psi : f_0 \cong f_1$ in $\mathcal{C}$ such that $\phi_1 (\psi \circ y_0) = \phi_0$.

\begin{definition}
  \label{def:weak-ext}
  Let $\mathcal{C}$ be an lcc category and $x \in \operatorname{Ob} \mathcal{C}$.
  A \emph{weak context extension} of $\mathcal{C}$ by $x$ consists of an lcc functor $f : \mathcal{C} \rightarrow \mathcal{D}$ and a morphism $v : t \rightarrow f(x)$ with $t$ a terminal object in $\mathcal{D}$ such that the following biuniversal property holds:

  For every lcc category $\mathcal{E}$, lcc functor $g : \mathcal{C} \rightarrow \mathcal{E}$ and morphism $w : u \rightarrow g(x)$ in $\mathcal{E}$ with $u$ terminal, the full subgroupoid of $\mathrm{Lcc}_{\mathcal{C} /}(f, g)$ given by pairs of lcc functor $h : \mathcal{D} \rightarrow \mathcal{E}$ and natural isomorphism $\phi : hf \cong g$ such that the square
  \begin{equation}
    \begin{tikzcd}
      h(t) \arrow[d] \arrow[r, "h(v)"] & h(f(x)) \arrow[d, "\phi_x"] \\
      u \arrow[r, "w"] & g(x)
    \end{tikzcd}
  \end{equation}
  in $\mathcal{D}$ commutes is contractible (i.e.\@ equivalent to the terminal groupoid).
\end{definition}

\begin{remark}
  \label{rem:discrete-weak-ext-mapping}
  Note that the definition entails that mapping groupoids of lcc functors $\mathcal{D} \rightarrow \mathcal{E}$ under $\mathcal{C}$ with $\mathcal{D}$ a weak context extension are equivalent to discrete groupoids.
  Lcc functors $h_0, h_1 : \mathcal{D} \rightarrow \mathcal{E}$ under $\mathcal{C}$ are (necessarily uniquely) isomorphic under $\mathcal{C}$ if and only if they correspond to the same morphism $w : u \rightarrow g(x)$ in $\mathcal{E}$.
\end{remark}

\begin{lemma}
  \label{lem:strictification}
  Let $\lambda : \Gamma \rightarrow F(G(\Gamma))$ be an $FG$-coalgebra and let $\Delta$ be a strict lcc category.
  Then the full and faithful inclusion of groupoids
  \begin{equation}
    \label{eq:forget-strictness}
    \mathrm{sLcc}(\Gamma, \Delta) \subseteq \mathrm{Lcc}(G(\Gamma), G(\Delta))
  \end{equation}
  admits a canonical retraction $f \mapsto f^s$.
  There is a natural isomorphism $\zeta^f : G(f^s) \cong f$, exhibiting the retract \eqref{eq:forget-strictness} as an equivalence of groupoids.
  The retraction $f \mapsto f^s$ and natural isomorphism $\zeta^f$ is $\mathrm{Gpd}$-natural in $(\Gamma, \lambda)$ and $\Delta$.
\end{lemma}
\begin{proof}
  Let $f : G(\Gamma) \rightarrow G(\Delta)$.
  The transpose of $f$ is a strict lcc functor $\bar f : F(G(\Gamma)) \rightarrow \Delta$ such that $G(\bar f) \eta = f$.
  We set $f^s = \bar f \lambda$ and $\zeta^f = G(\bar f) \phi$ for $\phi : G(\lambda) \cong \eta$ as in Lemma \ref{lem:coalgebra-vs-eta}.
  If $f = G(g)$ already arises from a strict lcc functor $g : \Gamma \rightarrow \Delta$, then $\bar g = g \varepsilon$ and hence $\bar g \lambda = g$.
  The action of the retraction $f \mapsto f^s$ on natural isomorphisms $f_0 \cong f_1$ is defined analogously from the $\mathrm{Gpd}$-enrichment of $F \dashv G$.
\end{proof}

\begin{lemma}
  \label{lem:strict-ext-is-weak-ext}
  Let $(\Gamma, \lambda)$ be an $FG$-coalgebra.
  Then $G(p) : G(\Gamma) \rightarrow G(\Gamma.\sigma)$ and $v : 1 \rightarrow p(\sigma)$ form a weak context extension of $G(\Gamma)$ by $\sigma$.
\end{lemma}
\begin{proof}
  Let $f : G(\Gamma) \rightarrow \mathcal{E}$ be an lcc functor and $w : t \rightarrow  f(\sigma)$ be a morphism with terminal domain in $\mathcal{E}$.
  Let $\Delta$ be a strict lcc category such that $G(\Delta) = \mathcal{E}$.
  Then by Lemma \ref{lem:strictification} there is an isomorphism $\zeta^f : G(f^s) \cong f$ for some strict lcc functor $f^s : \Gamma \rightarrow \Delta$.
  Set $g = \langle f^s, w^s \rangle$, where $w^s$ is the unique morphism in $G(\Delta)$ such that
  \begin{equation}
    \begin{tikzcd}
      1 \arrow[d] \arrow[r, "w^s"] & f^s(\sigma) \arrow[d, "\zeta^f_\sigma"] \\
      t \arrow[r, "w"] & f(\sigma)
    \end{tikzcd}
  \end{equation}
  commutes.
  (Both vertical arrows are isomorphisms.)
  Now with $g = \langle f^s, w^s \rangle : \Gamma.\sigma \rightarrow \Delta$ we have $\zeta^f : G(g) \circ G(p) \cong f$.

  Let $h : G(\Gamma.\sigma) \rightarrow \mathcal{E}$ and $\phi : h \circ G(p) \cong f$ be any other lcc functor over $G(\Gamma)$ such that
  \begin{equation}
    \begin{tikzcd}
      h(1) \arrow[d] \arrow[r, "h(v)"] & h(\sigma) \arrow[d, "\phi_\sigma"] \\
      t \arrow[r, "w"] & f(\sigma)
    \end{tikzcd}
  \end{equation}
  commutes.
  We need to show that $h$ and $G(g)$ are uniquely isomorphic under $G(\Gamma)$.
  Lemma \ref{lem:strictification} reduces this to the unique existence of an extension of the isomorphism $g p \cong h^s p : \Gamma \rightarrow \Delta$ defined as composite
  \begin{equation}
    G(gp) \cong f \cong h \circ G(p) \cong G((h \circ G(p))^s) = G(h^s p)
  \end{equation}
  to an isomorphism $g \cong h^s : \Gamma.\sigma \rightarrow \Delta$ under $\Gamma$.
  This follows from the construction of $\Gamma.\sigma$ as pushout
  \begin{equation}
    \begin{tikzcd}
      F(\{ t, \sigma \}) \arrow[d] \arrow[r] & F(\{ v : t \rightarrow \sigma \}) \arrow[d] \\
      \Gamma \arrow[r, "p"] & \Gamma.\sigma,
    \end{tikzcd}
  \end{equation}
  and its universal property on 2-cells.
\end{proof}

\begin{lemma}
  \label{lem:slice-is-weak-ext}
  Let $x$ be an object of an lcc category $\mathcal{C}$, and let $x^* : \mathcal{C} \rightarrow \mathcal{C}_{/ x}$ be any choice of pullback functor.
  Denote by $d = \langle \mathrm{id}_x, \mathrm{id}_x \rangle : \mathrm{id}_x \rightarrow x^*(x)$ the diagonal morphism in $\mathcal{C}_{/ x}$.
  Then $x^*$ and $d$ form a weak context extension of $\mathcal{C}$ by $x$.
\end{lemma}
\begin{proof}
  Let $\mathcal{E}$ be an lcc category, $f : \mathcal{C} \rightarrow \mathcal{E}$ be an lcc functor and $w : t \rightarrow f(\sigma)$ be a morphism in $\mathcal{E}$ with $t$ terminal.
  We define the induced lcc functor $g : \mathcal{C}_{/ x} \rightarrow \mathcal{E}$ as composition
  \begin{equation}
    \begin{tikzcd}
      \mathcal{C}_{/ x} \arrow[r, "f_{/ x}"] & \mathcal{E}_{/ f(x)} \arrow[r, "w^*"] & \mathcal{E}
    \end{tikzcd}
  \end{equation}
  where $w^* : \mathcal{E}_{ / f(x)} \rightarrow \mathcal{E}_{/ t} \xrightarrow{\sim} \mathcal{E}$ is given by a choice of pullback functor.

  Let $y \in \operatorname{Ob} \mathcal{C}$.
  We denote the composite $f(x) \rightarrow t \xrightarrow{w} f(x)$ by $w'$.
  Then the two squares
  \begin{mathpar}
    \begin{tikzcd}
      g(x^*(y)) \arrow[r] \arrow[d] & f(x \times y) \arrow[d, "f(\mathrm{pr}_1)"] \\
      t \arrow[r, "w"] & f(x)
    \end{tikzcd}
    \and
    \begin{tikzcd}
      f(y) \arrow[r, "{\langle w, \mathrm{id} \rangle}"] \arrow[d] & f(x \times y) \arrow[d, "f(\mathrm{pr}_1)"] \\
      t \arrow[r, "w"] & f(x)
    \end{tikzcd}
  \end{mathpar}
  are both pullbacks over the same cospan.
  Here $\mathrm{pr}_1 = x^*(y)$ denotes the first projection of the product defining the pullback functor $x^*$, and $x \times y$ is the projection's domain.
  (These should not be confused with canonical products in strict lcc categories; $\mathcal{C}$ and $\mathcal{D}$ are only lcc categories.)
  $f$ preserves pullbacks, so $f(x \times y)$ is a product of $f(x)$ with $f(y)$.
  We obtain natural isomorphisms $\phi_y : g(x^*(y)) \cong f(y)$ relating the two pullbacks for all $y$.

  The diagram
  \begin{equation}
    \begin{tikzcd}
      t \arrow[r, "w"] \arrow[d, "w"] & f(x) \arrow[d, "{\langle w', \mathrm{id}\rangle}"] \arrow[r] & t \arrow[d, "w"] \\
      f(x) \arrow[r, "f(d)"] & f(x \times x) \arrow[r, "\mathrm{pr}_1"] & f(x)
    \end{tikzcd}
  \end{equation}
  commutes, and in particular the left square commutes.
  It follows that $\phi$ is compatible with $d$ and $w$.

  $g$ and $\phi$ are unique up to unique isomorphism because for every morphism $k : y \rightarrow x$ in $\mathcal{C}$, i.e.\@ object of $\mathcal{C}_{/ x}$, the square
  \begin{equation}
    \begin{tikzcd}
      k \arrow[r, "{\langle k, \mathrm{id} \rangle}"] \arrow[d, "k"] & x^*(y) \arrow[d, "x^*(k)"] \\
      \mathrm{id}_x \arrow[r, "d"] & x^*(x)
    \end{tikzcd}
  \end{equation}
  is a pullback square in $\mathcal{C}_{/ x}$.
\end{proof}

\begin{lemma}
  \label{lem:extension-vs-slice}
  Let $\lambda : \Gamma \rightarrow F(G(\Gamma))$ be an $FG$-coalgebra and let $\Gamma \vdash \sigma$ be a type.
  Then $G(p) : G(\Gamma) \rightarrow G(\Gamma.\sigma)$ and $\sigma^* : G(\Gamma) \rightarrow G(\Gamma_{/ \sigma})$ are equivalent objects of the coslice category $\mathrm{Lcc}_{G(\Gamma) /}$.
  The equivalence $a : G(\Gamma.\sigma) \rightleftarrows G(\Gamma_{/ \sigma}) : b$ can be constructed naturally in $(\Gamma, \lambda)$ and $\sigma$, in the sense that coalgebra morphisms in $(\Gamma, \lambda)$ preserving $\sigma$ induce natural transformations of diagrams
  \begin{equation}
    \label{eq:ext-equivalence}
    \begin{tikzcd}
      G(\Gamma.\sigma)^\mathcal{I} & G(\Gamma.\sigma) \arrow[r, shift left, "a"] \arrow[l] & G(\Gamma_{/ \sigma}) \arrow[l, shift left, "b"] \arrow[r] & G(\Gamma_{/ \sigma})^\mathcal{I}.
    \end{tikzcd}
  \end{equation}
\end{lemma}
\begin{proof}
  It follows immediately from Lemmas \ref{lem:strict-ext-is-weak-ext} and \ref{lem:slice-is-weak-ext} that $G(\Gamma.\sigma)$ and $G(\Gamma_{/ \sigma})$ are equivalent under $G(\Gamma)$.
  However, a priori the corresponding diagrams \eqref{eq:ext-equivalence} can only be assumed to vary pseudonaturally in $(\Gamma, \lambda)$ and $\sigma$, meaning that for example the square
  \begin{equation}
    \label{eq:strict-ext-to-slice}
    \begin{tikzcd}
      G(\Gamma.\sigma) \arrow[r] \arrow[d] & G(\Gamma_{/ \sigma}) \arrow[d] \\
      G(\Delta.f(\sigma)) \arrow[r] & G(\Delta_{/ f(\sigma)}) \\
    \end{tikzcd}
  \end{equation}
  induced by a coalgebra morphism $f : (\Gamma, \lambda) \rightarrow (\Delta, \mu)$ would only commute up to isomorphism.

  The issue is that Definition \ref{def:weak-ext} only requires that certain mapping groupoids are contractible to a point, but the choice of point is not uniquely determined.
  To obtain a square \eqref{eq:strict-ext-to-slice} that commutes up to equality, we have to explicitly construct a map $G(\Gamma.\sigma) \rightarrow G(\Gamma_{/ \sigma})$ (i.e.\@ point of the contractible mapping groupoid) and show that this choice is strictly natural.

  The map $G(\Gamma.\sigma) \rightarrow G(\Gamma_{/ \sigma})$ over $G(\Gamma)$ is determined up to unique isomorphism by compatibility with the (canonical) pullback functor $\sigma^* : G(\Gamma) \rightarrow G(\Gamma_{/ \sigma})$ and the diagonal $d : \mathrm{id}_\sigma \rightarrow \sigma^*(\sigma)$.
  Recall from the proof of Lemma \ref{lem:strict-ext-is-weak-ext} that $a = \langle (\sigma^*)^s, d^s \rangle : G(\Gamma.\sigma) \rightarrow G(\Gamma_{/ \sigma})$ and $\alpha = \zeta^{\sigma^*} : G(a) \circ G(p) \cong \sigma^*$ is a valid choice.
  $d$ is stable under strict lcc functors, hence by Lemmas \ref{prop:strict-slicing} and \ref{lem:strictification}, $a$ and $\alpha$ are natural in $FG$-coalgebra morphisms.

  As in the proof of Lemma \ref{lem:slice-is-weak-ext}, the map in the other direction can be constructed as composite
  \begin{equation}
    \begin{tikzcd}
      b : G(\Gamma_{/ \sigma}) \arrow[r, "p_{/ \sigma}"] & G(\Gamma.\sigma_{/ p(\sigma)}) \arrow[r, "v^*"] & G(\Gamma.\sigma_{/ 1}) \arrow[r, "\cong"] & G(\Gamma.\sigma),
    \end{tikzcd}
  \end{equation}
  where $v^*$ is the canonical pullback along the variable $v$, and the components of the natural isomorphism $\beta : b \sigma^* \cong G(p)$ are the unique isomorphisms relating pullback squares
  \begin{mathpar}
    \begin{tikzcd}
      b(\sigma^*(\tau)) \arrow[r] \arrow[d] & p(\sigma) \times p(\tau) \arrow[d, "\mathrm{pr}_1"] \\
      1 \arrow[r, "v"] & p(\sigma)
    \end{tikzcd}
    \and
    \begin{tikzcd}
      p(\tau) \arrow[r, "{\langle v, \mathrm{id} \rangle}"] \arrow[d] & p(\sigma) \times p(\tau) \arrow[d, "\mathrm{pr}_1"] \\
      1 \arrow[r, "v"] & p(\sigma).
    \end{tikzcd}
  \end{mathpar}
  All data involved in the construction are natural in $\Gamma$ by Proposition \ref{prop:strict-slicing}, hence so are $b$ and $\beta$.

  By Remark \ref{rem:discrete-weak-ext-mapping}, the natural isomorphisms $(b, \beta) \circ (a, \alpha) \cong \mathrm{id}$ and $\mathrm{id} \cong (b, \beta) \circ (a, \alpha)$ over $G(\Gamma)$ are uniquely determined given their domain and codomain.
  Their naturality in $(\Gamma, \lambda)$ and $\sigma$ thus follows from that of $(a, \alpha)$ and $(b, \beta)$.
\end{proof}

\begin{lemma}
  \label{lem:substitution-vs-pullback}
  Let $\lambda : \Gamma \rightarrow F(G(\Gamma))$ be an $FG$-coalgebra, let $\sigma, \tau$ be types in context $\Gamma$ and let $\Gamma.\tau \vdash t : p_\tau(\sigma)$ be a term.
  Let $\bar t : \tau \rightarrow \sigma$ be the morphism in $\Gamma$ that corresponds to $t$ under the isomorphism
  \begin{equation}
    \mathrm{Hom}_{\Gamma.\tau}(1, p_\tau(\sigma)) \cong \mathrm{Hom}_{\Gamma_{/ \tau}}(\mathrm{id}_\tau, \tau^*(\sigma)) \cong \mathrm{Hom}_\Gamma(\tau, \sigma)
  \end{equation}
  induced by the equivalence of Lemma \ref{lem:extension-vs-slice} and the adjunction $\Sigma_\tau \dashv \tau^*$.
  Then the square
  \begin{equation}
    \begin{tikzcd}
      G(\Gamma.\sigma) \arrow[r] \arrow[d, "{G(\langle p_\tau, s \rangle)}"'] & G(\Gamma_{/ \sigma}) \arrow[d, "\bar t^*"] \\
      G(\Gamma.\tau) \arrow[r] & G(\Gamma_{/ \tau}) \\
    \end{tikzcd}
  \end{equation}
  in $\mathrm{Lcc}_{G(\Gamma)/ }$ commutes up to a unique natural isomorphism that is compatible with $FG$-coalgebra morphisms in $(\Gamma, \lambda)$.
\end{lemma}
\begin{proof}
  $\bar t^*$ maps the diagonal of $\sigma$ to the diagonal of $\tau$ up to the canonical isomorphism $\bar t^* \circ \sigma^* \cong \tau^*$, hence Lemma \ref{lem:strict-ext-is-weak-ext} applies.
\end{proof}

\begin{theorem}
  \label{th:lcc-supports-types}
  The cwf $\operatorname{Coa} \mathrm{sLcc}$ is a model of dependent type theory with finite product, extensional equality, dependent product and dependent sum types.
\end{theorem}
\begin{proof}
  $\operatorname{Coa} \mathrm{sLcc}$ has an empty context and context extensions by Proposition \ref{prop:coa-slcc-exts}.
  Finite product and equality types are interpreted as in $\mathrm{sLcc}$ (see Proposition \ref{prop:slcc-model}).

  Let $\Gamma \vdash \sigma$ and $\Gamma.\sigma \vdash \tau$.
  Denote by $a : \Gamma.\sigma \rightarrow \Gamma_{/ \sigma}$ the functor that is part of the equivalence established in Lemma \ref{lem:extension-vs-slice}.
  Then $\Gamma \vdash \mathbf{\Sigma}_\sigma \, \tau$ respectively $\Gamma \vdash \mathbf{\Pi}_\sigma \, \tau$ are defined by application of the functors
  \begin{equation}
    \begin{tikzcd}
      \Gamma.\sigma \arrow[r, "a"] & \Gamma_{/ \sigma} \arrow[r, shift left, "\Sigma_\sigma"] \arrow[r, shift right, "\Pi_\sigma"'] & \Gamma
    \end{tikzcd}
  \end{equation}
  to $\tau$.

  $a$ being an equivalence and the adjunction $\sigma^* \dashv \Pi_\sigma$ establish an isomorphism
  \begin{equation}
    \mathrm{Hom}_{\Gamma.\sigma}(1, \tau) \cong \mathrm{Hom}_{\Gamma_{/ \sigma}}(\sigma^*(1), a(\tau)) \cong \mathrm{Hom}_{\Gamma}(1, \Pi_\sigma(a(\tau)))
  \end{equation}
  by which we define lambda abstraction $\Gamma \vdash \lambda(t) : \mathbf{\Pi}_\sigma \, \tau$ for some term $\Gamma.\sigma \vdash t : \tau$ and the inverse to $\lambda$ (i.e.\@ application of $p_\sigma(u)$ to the variable $\Gamma.\sigma \vdash v : \sigma$ for some term $\Gamma \vdash u : \mathbf{\Pi}_\sigma \, \tau$).

  Now let $\Gamma \vdash s : \sigma$ and $\Gamma \vdash t : \langle \mathrm{id}_\Gamma, s \rangle(\tau)$.
  The pair term $u = (s, t)$ of type $\Gamma \vdash \mathbf{\Sigma}_\sigma \, \tau$ is defined by the diagram
  \begin{equation}
    \begin{tikzcd}
      \langle \mathrm{id}_\Gamma, s \rangle(\tau) \arrow[r, "\cong"] & s^*(a(\tau)) \arrow[r] \arrow[d] & \Sigma_\sigma(a(\tau)) \arrow[d, "a(\tau)"] \\
      & 1 \arrow[ul, "t"] \arrow[r, "s"] \arrow[ur, "u"] & \sigma.
    \end{tikzcd}
  \end{equation}
  Here the isomorphism $\langle \mathrm{id}, s \rangle(\tau) \cong s^*(a(\tau))$ is a component of the natural isomorphism $\langle \mathrm{id}, s \rangle \cong s^* \circ a$ constructed in Lemma \ref{lem:substitution-vs-pullback}, instantiated for $\tau = 1$.
  Given just $u$ we recover $s$ by composition with $a(\tau)$, and then $t$ as composition
  \begin{equation}
    \begin{tikzcd}[column sep=large]
      1 \arrow[r, "{\langle \mathrm{id}, u \rangle}"] & s^*(a(\tau)) \arrow[r, "\cong"] & \langle \mathrm{id}_\Gamma, s \rangle(\tau).
    \end{tikzcd}
  \end{equation}
  These constructions establish an isomorphism of terms $s$ and $t$ with terms $u$, so the $\beta$ and $\eta$ laws hold.

  The functors $a, \sigma^*, \Sigma_\sigma, \Pi_\sigma$ and the involved adjunctions are preserved by $FG$-coalgebra morphisms (Proposition \ref{prop:strict-slicing}, Lemmas \ref{lem:extension-vs-slice} and \ref{lem:substitution-vs-pullback} ), so our type theoretic structure is stable under substitution.
\end{proof}

\section{Cwf structure on individual lcc categories}
\label{sec:applications}

In this section we show that the covariant cwf structure on $\operatorname{Coa} \mathrm{sLcc}$ that we established in Theorem \ref{th:lcc-supports-types} can be used as a coherence method to rectify Seely's interpretation in a given lcc category $\mathcal{C}$.

\begin{lemma}
  \label{lem:cats-of-weak-exts}
  Let $\lambda : \Gamma \rightarrow F(G(\Gamma))$ be an $FG$-coalgebra.
  Then the following categories are equivalent:
  \begin{enumerate}[label={(\arabic*)}]
    \item
      \label{itm:locally-discrete-gamma-op}
      $\Gamma^\mathrm{op}$;
    \item
      \label{itm:slice-2-category}
      the category of isomorphism classes of morphisms in the restriction of the higher coslice category $\mathrm{Lcc}_{G(\Gamma) /}$ to slice categories $\sigma^* : G(\Gamma) \rightarrow G(\Gamma_{/ \sigma})$;
    \item
      \label{itm:ext-2-category}
      the category of isomorphism classes of morphisms in the restriction of the higher coslice category $\mathrm{Lcc}_{G(\Gamma) /}$ to context extensions $G(p_\sigma) : G(\Gamma) \rightarrow G(\Gamma.\sigma)$;
    \item
      \label{itm:ext-1-category}
      the full subcategory of the 1-categorical coslice category $(\operatorname{Coa} \mathrm{Lcc})_{(\Gamma, \lambda) /}$ given by the context extensions $p_\sigma : (\Gamma, \lambda) \rightarrow (\Gamma.\sigma, \lambda.\sigma)$.
  \end{enumerate}
\end{lemma}
\begin{proof}
  As noted in Remark \ref{rem:discrete-weak-ext-mapping}, the higher categories in \ref{itm:slice-2-category} and \ref{itm:ext-2-category} are already locally equivalent to discrete groupoids and hence biequivalent to their categories of isomorphism classes.

  The functor from \ref{itm:locally-discrete-gamma-op} to \ref{itm:slice-2-category} is given by assigning to a morphism $s : \tau \rightarrow \sigma$ in $\Gamma$ the isomorphism class of the pullback functor $s^* : G(\Gamma_{/ \sigma}) \rightarrow G(\Gamma_{/ \tau})$.
  The isomorphism class of an lcc functor $f : G(\Gamma_{/ \sigma}) \rightarrow G(\Gamma_{/ \tau})$ over $G(\Gamma)$ is uniquely determined by the morphism
  \begin{equation}
    \begin{tikzcd}
      \mathrm{id}_\tau \arrow[r, "\cong"] & f(\mathrm{id}_\sigma) \arrow[r, "f(d)"] & f(\sigma^*(\sigma)) \arrow[r, "\cong"] & \tau^*(\sigma),
    \end{tikzcd}
  \end{equation}
  which in turn corresponds to a morphism $s : \tau = \Sigma_\tau \mathrm{id}_\tau \rightarrow \sigma$, and then $f \cong s^*$.

  The categories \ref{itm:slice-2-category} and \ref{itm:ext-2-category} are equivalent because they are both categories of weak context extensions (Lemmas \ref{lem:strict-ext-is-weak-ext} and \ref{lem:slice-is-weak-ext}).
  Finally, the inclusion of \ref{itm:ext-1-category} into \ref{itm:ext-2-category} is an equivalence by the Lemma \ref{lem:strictification}.
  Note that every strict lcc functor $\Gamma.\sigma \rightarrow \Gamma.\tau$ commuting (up to equality) with the projections $p_\sigma$ and $p_\tau$ is compatible with the coalgebra structures of $\lambda.\sigma : \Gamma \rightarrow \Gamma.\sigma$ and $\lambda.\tau : \Gamma \rightarrow \Gamma.\tau$.
\end{proof}

\begin{definition}
  Let $\mathcal{C}$ be a covariant cwf and let $\Gamma$ be a context of $ \mathcal{C}$.
  Then the \emph{coslice covariant cwf} $\mathcal{C}_{\Gamma /}$ has as underlying category the (1-categorical) coslice category under $\Gamma$, and its types and terms are given by the composite functor $\mathcal{C}_{\Gamma /} \xrightarrow{\mathrm{cod}} \mathcal{C} \rightarrow \mathrm{Fam}$.
\end{definition}

\begin{lemma}
  \label{lem:coslice-cwf}
  Let $\mathcal{C}$ be a covariant cwf and let $\Gamma$ be a context of $\mathcal{C}$.
  Then the coslice covariant cwf $\mathcal{C}_{\Gamma /}$ has an initial context.
  If $\mathcal{C}$ has context extensions, then $\mathcal{C}_{\Gamma /}$ has context extensions, and they are preserved by $\mathrm{cod} : \mathcal{C}_{\Gamma /} \rightarrow \mathcal{C}$.
  If $\mathcal{C}$ supports any of finite product, extensional equality, dependent product or dependent sum types, then so does $\mathcal{C}_{\Gamma /}$, and they are preserved by $\mathrm{cod} : \mathcal{C}_{\Gamma /} \rightarrow \mathcal{C}$
  \qed
\end{lemma}

\begin{definition}
  Let $\mathcal{C}$ be a covariant cwf with an empty context and context extensions.
  The \emph{core} of $\mathcal{C}$ is a covariant cwf on the least full subcategory $\operatorname{Core} \mathcal{C} \subseteq \mathcal{C}$ that contains the empty context and is closed under context extensions, with types and terms given by $\operatorname{Core} \mathcal{C} \hookrightarrow \mathcal{C} \rightarrow \mathrm{Fam}$.
\end{definition}

\begin{lemma}
  \label{lem:core-cwf}
  Let $\mathcal{C}$ be a covariant cwf with an empty context and context extension.
  If $\mathcal{C}$ supports any of finite product types, extensional equality types, dependent product or dependent sums, then so does $\operatorname{Core} \mathcal{C}$, and they are preserved by the inclusion $\operatorname{Core} \mathcal{C} \hookrightarrow \mathcal{C}$.

  If $\mathcal{C}$ supports unit and dependent sum types, then $\operatorname{Core} \mathcal{C}$ is democratic, i.e.\@ every context is isomorphic to a context obtained from the empty context by a single context extension \citep{biequivalence-lcc-cwf}.
  \qed
\end{lemma}

\begin{theorem}
  \label{th:context-as-model}
  Let $\lambda : \Gamma \rightarrow F(G(\Gamma))$ be an $FG$-coalgebra.
  Then the underlying category of $\operatorname{Core} \, ((\operatorname{Coa} \mathrm{sLcc})_{(\Gamma, \lambda) /})$ is equivalent to $U(G(\Gamma))^\mathrm{op}$.
  In particular, every lcc category is equivalent to a cwf that has an empty context and context extensions, and that supports finite product, extensional equality, dependent sum and dependent product types.
\end{theorem}
\begin{proof}
  $\operatorname{Core} \, ((\operatorname{Coa} \mathrm{sLcc})_{(\Gamma, \lambda) /})$ is a covariant cwf supporting all relevant type constructors by Lemmas \ref{lem:coslice-cwf} and \ref{lem:core-cwf}.
  It is democratic and hence equivalent to category \ref{itm:ext-1-category} of lemma \ref{lem:cats-of-weak-exts}.

  Given an arbitrary lcc category $\mathcal{C}$, we set $\Gamma = F(\mathcal{C})$ and define coalgebra structure by $\lambda = F(\eta) : F(\mathcal{C}) \rightarrow F(G(F(\mathcal{C})))$.
  Then $G(\Gamma)$ is equivalent to both $\mathcal{C}$ and a cwf supporting the relevant type constructors.
\end{proof}

\section{Conclusion}

We have shown that the category of lcc categories is a model of extensional dependent type theory.
Previously only individual lcc categories were considered as targets of interpretations.
As in these previous interpretations, we have had to deal with the issue of coherence:
Lcc functors (and pullback functors in particular) preserve lcc structure only up to isomorphism, whereas substitution in type theory commutes with type and term formers up to equality.

Our novel solution to the coherence problem relies on working globally, on all lcc categories at once.
In contrast to some individual lcc categories, the higher category of all lcc categories is locally presentable.
This allows the use of model category theory to construct a presentation of this higher category in terms of a 1-category that admits an interpretation of type theory.

While we have only studied an interpretation of a type theory with dependent sum and dependent product, extensional equality and finite product types, it is straightforward to adapt the techniques of this paper to type theories with other type constructors.
For example, a dependent type theory with a type of natural numbers can be interpreted in the category of lcc categories with objects of natural numbers.
Alternatively, we can add finite coproduct, quotient and list types but omit dependent products, and obtain an interpretation in the category of arithmetic universes \citep{au-as-list-arithmetic-pretopos,au-sketches}.

I would expect there to be a general theorem by which one can obtain a type theory and its interpretation in the category of algebras for every (higher) monad $M$ on $\mathrm{Cat}$ (with the algebras of $M$ perhaps subject to being finitely complete and stable under slicing).
Such a theorem, however, is beyond the scope of the present paper.

\bibliographystyle{abbrvnat}
\setcitestyle{authoryear,open={(},close={)}}
\bibliography{main}

\end{document}